\documentclass[12pt,a4paper]{amsart}

\usepackage{amssymb,xcolor}

\usepackage{graphicx,amssymb,amsfonts,epsfig,amsthm,a4,amsmath,url,verbatim}
\usepackage[latin1]{inputenc}

\newtheorem{thm}{Theorem}[section]
\newtheorem{cor}[thm]{Corollary}
\newtheorem{lem}[thm]{Lemma}

\newtheorem{prop}[thm]{Proposition}

\newtheorem{fact}[thm]{Fact}

\theoremstyle{definition}
\newtheorem{defn}[thm]{Definition}

\newtheorem{que}[thm]{Question}
\newtheorem{conj}[thm]{Conjecture}
\newtheorem{exe}[thm]{Example}

\newtheorem{rem}[thm]{Remark}

\numberwithin{equation}{section}

\newcommand{\SL}{\textnormal{SL}}

\newcommand{\eps}{\varepsilon}
\newcommand{\tu}{\bigtriangleup}
\newcommand{\Wob}{\textnormal{Wob}}
\newcommand{\elk}{\mathfrak{L}}
\newcommand{\odFH}{$\overline{\text{dFH}}$}

\begin{document}
\title[Lattices, means, and commensurating actions]{Irreducible lattices, invariant means, and commensurating actions}
\author{Yves Cornulier}%
\address{Laboratoire de Math\'ematiques\\
B\^atiment 425, Universit\'e Paris-Sud 11\\
91405 Orsay\\FRANCE}
\email{yves.cornulier@math.u-psud.fr}
\thanks{Supported by ANR grant GSG 12-BS01-0003-01}

\date{June 24, 2014 (April 2020: added correction in Example \ref{exint}(\ref{exint2}), updated references)}

\begin{abstract}
We study rigidity properties of lattices in terms of invariant means and commensurating actions (or actions on CAT(0) cube complexes). We notably study Property FM for groups, namely that any action on a discrete set with an invariant mean has a finite orbit. 
\end{abstract}

\maketitle

\section{Introduction}

\subsection{Basic definitions}

Let $G$ be a topological group. We call {\em continuous discrete $G$-set} a discrete set $X$ endowed with a continuous action of $G$. Note that continuity of the action means that the point stabilizers are open subgroups of $G$. Recall that a subset $M\subset X$ is {\em commensurated} by the $G$-action if $$\elk_M(g)=\#(M\tu gM)<\infty,\qquad \forall g\in G,$$ where $\tu$ denotes the symmetric difference. This holds in particular when $M$ is {\em transfixed}, in the sense that there there exists a $G$-invariant subset $N$ {\em commensurate} to $M$ in the sense that $\#(M\tu N)<\infty$. Brailovsky, Pasechnik and Praeger \cite{BPP} have shown that $M$ is transfixed if and only if the function $\elk_M$ is bounded on $G$ (see also Corollary \ref{cabou}).

\begin{defn}[\cite{CorFW}]
A locally compact group $G$ has {\em Property FW} if for every continuous discrete $G$-set $X$, every commensurated subset $M\subset X$ is transfixed.
\end{defn}

We also use the following terminology:

\begin{defn}
A locally compact group $G$ has {\em Property FM} if for every continuous discrete $G$-set $X$ with a $G$-invariant mean on all subsets of $X$, there exists a finite $G$-orbit.
\end{defn}

(Property FM for countable discrete groups appears as the negation of ``being in the class ($\mathcal{B}$)" in \cite{GM}.)

Property FW has various characterizations, including:
\begin{itemize}
\item every continuous cellular action on any CAT(0) cube complex has bounded orbits for the $\ell^1$-metric;
\item every continuous cellular action on any CAT(0) cube complex has a fixed point;

\item (if $G$ is compactly generated) for every open subgroup of infinite index of $G$, the Schreier graph of $G/H$ is 1-ended.
\end{itemize}

It is important to allow infinite-dimensional cube complexes in the above equivalences.

Property FW can be viewed as a strengthening of Serre's Property FA (every isometric action on a tree has a fixed point on the 1-skeleton) and a weakening of Property FH (every isometric action on a Hilbert space has a fixed point), which is equivalent to Kazhdan's Property~T for $\sigma$-compact locally compact groups. 

Let us consider one more related property:

\begin{defn}\label{ddfh}
A locally compact group $G$ has {\em Property dFH} if for every continuous discrete $G$-set $X$, we have $H^1(G,\ell^2(X))=0$.
\end{defn}

Here dFH stands for ``discrete Property FH". Obviously Property FH implies Property dFH, which, by standard arguments, implies both Properties FW and FM (the latter for $\sigma$-compact groups); see~\S\ref{red1c}. It also implies Property $\tau$, and, for discrete groups, implies the vanishing of the first $\ell^2$-Betti number.

\subsection{Generalities}

An extensive discussion on Property FW can be found in \cite{CorFW}. Let us provide some basic facts about Property FM. For convenience, we state them for locally compact groups, although most of them are stated for arbitrary topological groups in \S\ref{s_ame}. Let us first mention that Property FM can be viewed, at least for infinite discrete groups, as a strong form of non-amenability. However, many infinite discrete non-amenable groups fail to have Property FM, sometimes for trivial reasons (such as the existence of an infinite amenable quotient), or for more subtle reasons; see for instance the examples and the discussion in \cite{GN}.

\begin{prop}The class of locally compact groups with Property FM 
\begin{enumerate}
\item (Proposition \ref{fmfi}) is invariant by passing to and from open finite index subgroups, and is more generally inherited from closed cocompact subgroups;
\item (Fact \ref{squot}) is stable by taking dense images and in particular by taking quotients; 
\item (Proposition \ref{fmext}) is stable by taking extensions;
\item (Proposition \ref{fmcg}). If $G$ is a locally compact group with Property FM, then it is compactly generated. In particular, if $G$ is discrete then it is finitely generated.
\end{enumerate}
\end{prop}

Recall that a locally compact group $G$ has {\em Kazhdan's Property T} if every continuous unitary representation with almost invariant vectors has nonzero invariant vectors. Here, a unitary representation of $G$ {\em has almost invariant vectors} if for every compact subset $K\subset G$ and every $\eps>0$ there exists a vector $\xi$ with $\|\xi\|=1$ and $\sup_{s\in K}\|g\xi-\xi\|\le\eps$. A standard observation (see however the comments following Lemma \ref{eyml2}) is:

\begin{prop}[Proposition \ref{tfw}]
Let $G$ be a locally compact group with Property T. Then it has Property FM.
\end{prop}

Let us also mention a variant of \cite[Lemma 4.5]{GM}.

\begin{prop}[Proposition \ref{tars}]\label{itars}
Let $G$ be a non-amenable locally compact group in which every open subgroup of infinite index is amenable. Then $G$ has Property FM.
\end{prop}

\begin{defn}\label{algty}
Consider a locally compact group $G$ with a {\em topological almost direct product decomposition} $G=G_1...G_n$, meaning that the $G_i$ are non-compact closed normal subgroups centralizing each other, and the canonical homomorphism $G_1\times\dots \times G_n$ is proper, surjective with finite kernel.

We say that $G$ is {\em semisimple of algebraic type} (with respect to this decomposition) if each $G_i$ is topologically isomorphic to the group of $\mathbf{K}_i$-points of an almost $\mathbf{K}_i$-simple, semisimple $\mathbf{K}_i$-isotropic linear algebraic group and $\mathbf{K}_i$ is a non-discrete locally compact field. 
\end{defn}

Using the Howe-Moore property for each $G_i$ and Proposition \ref{itars}, we obtain:

\begin{prop}[Proposition \ref{ssfm}]\label{issfm}
If $G$ is semisimple of algebraic type, then it has Property FM.
\end{prop}

\subsection{Irreducible lattices}

The results and the discussion below are motivated by the following two conjectures.

\begin{conj}\label{conjs}
Let $S$ be a semisimple connected Lie group with at least two simple factors and no compact factor. Then every irreducible lattice $\Gamma$ in $S$ has Property FW.
\end{conj}

The weaker Property FA for these groups was proved by Margulis \cite{Mar81}.
A similar (and more perilous) conjecture can be stated for Property FM:

\begin{conj}\label{conjm}
Let $S$ be a semisimple connected Lie group with at least two simple factors and no compact factor. Then every irreducible lattice in $S$ has Property FM.
\end{conj}

The following theorem follows from Proposition \ref{pnarf} and Corollary \ref{cowm}.

\begin{thm}\label{iccj}
If $\Gamma$ is as in Conjecture \ref{conjs} and $X$ is a $\Gamma$-set with a commensurated subset not commensurate to any invariant subset, then $X$ has an invariant mean. In particular, Conjecture \ref{conjm} for $\Gamma$ implies Conjecture \ref{conjs} for $\Gamma$.
\end{thm}

Property T implies both Properties FW and FM and passes to lattices and therefore the conjecture obviously holds when $S$ has Property T. In fact, a weaker assumption is enough:

\begin{thm}[see Theorem \ref{somtfm}]\label{ipart}
If $S$ has at least one simple factor with Property T, then it satisfies both conjectures. 
\end{thm}

This follows from Theorem \ref{somtfm} for Property FM, combined with Theorem \ref{iccj} for Property FW (alternatively, for Property FW, a direct proof also follows from Theorem \ref{fwres}).
The Property FM part of this theorem was independently proved by Bekka and Olivier \cite{BO}.

Thus the hard case of the conjectures is when $S$ has the Haagerup Property, typically $\SL_2(\mathbf{R})^2$. Using bounded generation by unipotent elements, some non-cocompact lattices therein, e.g.\ $\SL_2(\mathbf{Z}[\sqrt{2}])$ are known to have Property FW \cite[Example 6.4]{CorFW} but Property FM is unknown (including the case of $\SL_2(\mathbf{Z}[\sqrt{2}])$, explicitly asked in \cite[4.I]{GM}), as well as Property FW for cocompact lattices.

Let us also mention that the method of Theorem \ref{ipart} (namely Theorem \ref{fwres}) also works in a non-Archimedean setting. However, the statement differs, because this makes a distinction between Property FW and FM. Theorem \ref{ipart} thus generalizes as:

\begin{thm}[See Theorem \ref{somtfm}]Let $G=G_1\dots G_n$ be a semisimple group of algebraic type (as in Definition \ref{algty}), with $n\ge 2$. Let $\Gamma$ be an irreducible lattice, in the sense that $G_i\Gamma$ is dense in $G$ for all $i$. Suppose that at least one of the $G_i$ has Property T. Then
\begin{itemize}
\item $\Gamma$ has Property FM, and
\item if moreover none of the $G_i$ is non-Archimedean of $\mathbf{K}_i$-rank~1, then $\Gamma$ has Property FW.
\end{itemize}
\end{thm}

\begin{exe}Fix a prime $p$. For any $\mathbf{Q}$-form $G$ of $\textnormal{SO}_5$, such that $G$ has $\mathbf{R}$-rank 2 and $\mathbf{Q}_p$-rank 1, consider the group $\Gamma=G(\mathbf{Z}[1/p])$. It is an irreducible lattice in $G(\mathbf{R})\times G(\mathbf{Q}_p)$. Then $\Gamma$ has Property FM; however it does not have Property FW and not even Property FA, as it is a dense subgroup of the noncompact group $G(\mathbf{Q}_p)$, which has a proper isometric action on a tree.
\end{exe}

The idea behind this example is that for irreducible lattices in products with reasonable hypotheses, Property FA or FM of the lattice is known or expected to follow from the same property for the ambient group. For instance, the previous example uses that if $S$ is simple of rank 1 over a (non-archimedean) local field, then $S$ has Property FM (Proposition \ref{issfm}) but not FW, while both of these properties are trivially satisfied by connected groups, which is a heuristic evidence towards the previous conjectures. This also shows the interest in defining this for locally compact groups and not only in the discrete setting; this will be used in the proof of Theorem \ref{iccj}.

\begin{rem}
It has been obtained by Chatterji, Fernos and Iozzi \cite{CFI}
that a group $\Gamma$ as in Conjecture \ref{conjs} admits no ``non-elementary" action on a {\em finite-dimensional} CAT(0) cube complex. (I put quotation marks because the terminology ``non-elementary" is misleading, as for this terminology elementary actions may contain non-elementary actions as subactions or quotient actions, and non-elementary should rather be interpreted as a kind of irreducibility assumption.) Pierre-Emmanuel Caprace indicated me (private communication) how this result can be used to prove the improved statement that such a group $\Gamma$ has no {\em unbounded} action on any finite-dimensional CAT(0) cube complex (this is planned to be an appendix to \cite{CFI}). This provides some further evidence for Conjecture \ref{conjs}.
\end{rem}

\begin{rem} It follows from a result of Napier and Ramachandran \cite{NR} that if $S$ is a semisimple group of rank $\ge 2$ and of Hermitian type (i.e.\ the associated symmetric space is Hermitian), and $\Gamma$ is an irreducible lattice in $S$ then every Schreier graph of $\Gamma$ has at most 2 ends. (Recall that a finitely generated group $G$ has Property FW if and only if every Schreier graph of $G$ has at most 1 end.)
\end{rem}

\subsection{Further results and questions}

I do not know any finitely generated group with Property FW but not FM, but this should certainly exist (as well as infinite amenable groups with Property FW, see Question \ref{quesfm}(\ref{qamfw})). Note that on the other hand, there exist uncountable discrete groups with Property FW but not FM (see Remark \ref{fwsb}).

There are natural weakenings FW' and FM' of Properties FW and FM, obtained by using the same definition but restricting to transitive actions. The question of finding a countable group with Property FW' but not FW (resp.\ FM' but not FM) is not straightforward; still it is solved in \S\ref{sl3q}:

\begin{thm}\label{slq}
The group $\SL_n(\mathbf{Q})$ for $n\ge 3$ has Properties FM' and FW' (but not Properties FM and FW).
\end{thm}

The proof of Theorem \ref{slq} makes use of the embedding of $\SL_n(\mathbf{Q})$ as a lattice in $\SL_n$ of the product of the ring of adeles with $\mathbf{R}$, and Property T for $\SL_n(\mathbf{R})$.

Some applications of Properties FM and FW to groups of permutations with bounded displacement are given in \S\ref{appfm}; a particular case is the following:

\begin{thm}
Let $\Gamma$ be a finitely generated group with Property FM or FW. Then any action of $\Gamma$ on $\mathbf{Z}$ by permutations of bounded displacement factors through a finite group.
\end{thm}

Let us end this introduction by some further questions.

\begin{que}\label{quesfm}~
\begin{enumerate}
\item Consider R.~Thompson's groups $T$ and $V$ of the circle and the Cantor set. Do $T$ and $V$ have Property FM?
\item\label{qamfw} Does there exist an infinite finitely generated amenable group with Property FW?
\item Does there exist a finitely generated group without Property FM, but for which every infinite Schreier graph has exponential growth? (For a finitely generated group with Property FM, every infinite Schreier graph has exponential growth.) 
\item Does there exist a finitely generated group with nonzero first $\ell^2$-Betti number and with Property FW?
\item\label{meana} Given a group action on a connected median graph (or equivalently on a CAT(0) cube complex), give a geometric characterization of the non-existence of an invariant mean on the set of proper halfspaces. 
\end{enumerate}
\end{que}

In the case of unbounded actions on trees, (\ref{meana}) has a simple answer: there is no mean if and only if the action is of general type, i.e.\ has no invariant axis, or point at infinity; this is used \cite[Proposition 9.1]{Sha} by Shalom to prove a superrigidity statement for actions on trees.

\medskip

\noindent {\bf Acknowledgements.} I thank Bachir Bekka for discussions and pointing out \cite{BO}; I also thank Alessandra Iozzi and Alain Valette for some useful references. I am grateful to the referee for pointing out many inaccuracies and for valuable suggestions.

\setcounter{tocdepth}{1}    
\tableofcontents

\section{Preliminaries}

\subsection{Affine $\ell^p$ action}\label{ala}

We here recall some classical material, following \cite{CorFW}. We fix $p\in [1,\infty\mathclose[$.  Fix a locally compact group $G$ with a continuous action on a discrete set $X$ and a commensurated subset $M\subset X$. 

Define \[\ell^p_M(X)=\{f\in\mathbf{R}^X:f-\mathbf{1}_M\in\ell^p(X)\}.\] It is endowed with a canonical structure of an affine space over $\ell^p(X)$ and the corresponding distance. It only depends on the commensuration class of $M$. That $M$ is commensurated implies that $\ell^p_M(X)$ is invariant under the natural action of $G$ on $\mathbf{R}^X$. The action of $G$ on $\ell^p_M(X)$ (endowed with the $\ell^p$-distance) is isometric and continuous; we have
\[\|\mathbf{1}_M-g\mathbf{1}_M\|_p^p=\#(M\tu gM),\quad \forall g\in G.\]

\subsection{$Q$-points}\label{uq}

We need some material essentially borrowed from \cite[Section 4]{Cor}. Let $H$ be a group (regardless of any topology on $H$), and we fix a homomorphism $\pi$ from $H$ to a Hausdorff topological group $Q$, with dense image.

Define $\mathcal{V}_Q$ as the set of subsets of $H$ containing $\pi^{-1}(V)$ for some neighborhood $V$ of 1 in $Q$.

\begin{defn}\label{xqq}
Let $H$ act by isometries on a metric space $D$.  We define the subset of $Q$-points in $D$ as
\[D^Q=\left\{x\in D:\;\inf_{V\in\mathcal{V}_Q}\sup_{g\in V}d(x,gx)=0\right\}.\]
\end{defn}

\begin{rem}\label{dqdis}
If $D$ is discrete then 
$$D^{Q}=\left\{x\in D:\;\exists V\in\mathcal{V}_Q: V\subset H_x\right\},$$ where $H_x\subset H$ is the stabilizer of $x$. Note that this does not depend on the discrete $H$-invariant distance on $D$.
\end{rem}

\begin{lem}\label{wcloq}
The subset $D^{Q}$ is closed in $D$ and $H$-invariant. If moreover $D$ is a normed real vector space and the action of $H$ is linear then $D$ is a closed linear subspace.
\end{lem}
\begin{proof}
Let $x_n\in D^{Q}$ converge to $x\in D$ and let us check that $x\in D^Q$.
Fix $\eps>0$. Let us fix $n$ such that $d(x_n,x)\le\eps$, so by the triangle inequality we get, for all $g\in H$
$$d(gx,x)\le d(gx,gx_n)+d(gx_n,x_n)+d(x_n,x)\le d(gx_n,x_n)+2\eps;$$
since $x_n\in D^{Q}$, there exists $V\in\mathcal{V}_Q$ such that for every $g\in V$ we have $d(x_n,gx_n)\le\eps$. So for all $g\in V$ we obtain $d(gx,x)\le 3\eps$. Thus $x\in D^{Q}$. So $D^{Q}$ is closed. (Note that the proof extends to the case when the action is by uniformly bilipschitz maps.)

The other assertions are clear.\end{proof}

The motivation of the previous definitions is the following fact \cite[Propositions 4.1.2 and 4.1.3]{Cor}.

\begin{prop}\label{cext}
Let $H$ act by isometries on a metric space $D$. Then the action of $H$ on $D^Q$ uniquely factors through a continuous action of $Q$.
\end{prop}
\begin{proof}[On the proof]
The uniqueness is clear, by density of the image of $\pi$ and using that $D$ is Hausdorff.

The existence is actually obtained in \cite[\S 4.1]{Cor} assuming $X$ complete, but the general case follows: let $Y$ be the completion of $D$, so that $Y^Q$ and $D$ are both $H$-invariant; the action on $Y^Q$ factors through a continuous action of $Q$ by the complete case. Then clearly $D^Q=Y^Q\cap D$, so the $Q$-action on $Y^Q$ restricts to a continuous action of $Q$ on $D^Q$.
\end{proof}

\begin{thm}[Theorem 4.7.4 in \cite{Cor}]\label{coexx}
Let $f:H\to Q$ be a continuous homomorphism with dense image between topological groups. Let $H$ act continuously by isometries on a complete CAT(0)
metric space $D$. Suppose that there exists a neighbourhood
$\Omega$ of $1$ in $Q$, such that, for some $w\in D$,
$f^{-1}(\Omega)w$ is bounded. Then $D^Q$ is nonempty.
\end{thm}
\begin{proof}[On the proof]
The statement \cite[Theorem 4.7.4]{Cor} assumes that $H,Q$ are locally compact, but it is not used in the proof (the first line of the proof considers a sequence $(\Omega_n)$ of compact subsets of $\Omega$ but $\Omega_n$ being compact is never used).
\end{proof}

\subsection{Kazhdan homomorphisms}\label{kh}

Let $\pi:H\to Q$ be a homomorphism between locally compact groups, with dense image. If $(u,\mathcal{H})$ is an orthogonal or unitary continuous representation of $H$, we denote by $u^Q$ the representation of $Q$ on $\mathcal{H}^Q$, which exists and is continuous by Proposition \ref{cext}.

\begin{defn}[{\cite[\S 4.2]{Cor}}]
We say that $\pi$ is a {\em Kazhdan homomorphism} if for every continuous orthogonal representation $u$ of $H$ such that $\mathbf{1}_H\prec u$ (i.e.\ $u$ almost has invariant vectors), we have $\mathbf{1}_Q\prec u^Q$.
\end{defn}

Kazhdan homomorphisms are called {\em resolutions} in \cite{Cor} but this choice of terminology is questionable. 

Note that the requirement $\mathbf{1}_Q\prec u^Q$ in particular implies $u^Q\neq 0$. Actually, if we modify the conclusion $\mathbf{1}_Q\prec u^Q$ into the weaker conclusion $u^Q\neq 0$, we obtain, at least for $\sigma$-compact groups, an equivalent definition, but the proof is not straightforward; see \cite[\S 2.3.8]{Cor05}. 

\begin{exe}\label{exkho}
The identity map of any locally compact group is always Kazhdan. The trivial homomorphism $H\to 1$ is Kazhdan if and only if $H$ has Kazhdan's Property T. More generally, if $N$ is a closed normal subgroup of $H$, then the quotient homomorphism $H\to H/N$ is Kazhdan if and only $(H,N)$ has relative Property T. Besides, if $H$ is Haagerup and $H\to Q$ is Kazhdan, then it is the quotient homomorphism by a compact subgroup.
\end{exe}

The following theorem, which generalized \cite[Theorem III.6.3]{Mar} and \cite{BL}, provides examples of Kazhdan homomorphisms beyond quotient homomorphisms.

\begin{thm}[{\cite[Theorem 4.3.1]{Cor}}]\label{mbl}
Let $G$ be a locally compact group, $N$ a normal subgroup such that $(G,N)$ has relative Property T. Let $H$ be a closed, finite covolume subgroup of $G$ whose projection in $Q=G/N$ is dense. Then the projection $H\to Q$ is a Kazhdan homomorphism.
\end{thm}

Kazhdan homomorphisms are useful to transfer various rigidity properties of $Q$ back to $H$. For instance, $Q$ has Property $(\tau)$ if and only if $H$ has Property $(\tau)$, $Q$ is compactly generated if and only if $H$ is compactly generated, see \cite[Theorem 4.2.8, Proposition 4.5.2]{Cor} for this and more examples as well as Theorems \ref{fwres} and \ref{fmres} in this paper.

\begin{thm}[{\cite[Theorem 4.7.6]{Cor}}]\label{afr}
Let $\pi:H\to Q$ be a Kazhdan homomorphism with dense image between locally compact groups. Then for every continuous affine isometric action of $H$ on a Hilbert space $V$, we have $V^Q\neq\emptyset$. 
\end{thm}

This has the following consequence:

\begin{cor}\label{surjh1}
Let $\pi:H\to Q$ be a Kazhdan homomorphism with dense image between locally compact groups. Then for every continuous orthogonal representation of $H$, the restriction maps $H^1(Q,u^Q)\to H^1(H,u)$ and $\overline{H^1}(Q,u^Q)\to \overline{H^1}(H,u)$ are well-defined and surjective.
\end{cor}
\begin{proof}
That these maps are well-defined is straightforward.

Indeed, given $b\in Z^1(H,u)$, then some cohomologous cocycle $b'$ has 0 as a $Q$-point for the affine action defined by $b'$, which means that $b'=c\circ\pi$ for some $c\in Z^1(Q,u^Q)$. The statement in reduced cohomology follows.
\end{proof}

\subsection{Some maps}\label{soma}

If $Z,Y$ are sets and $\eta:Z\to Y$ is a surjective map with finite fibers, we denote by $\ell^p_{[Y]}(Z)$ the set of functions in $\ell^p(Z)$ that are constant on each fiber of $\eta$ (rigorously speaking, we should rather write $\ell^p_\eta(Z)$). Writing $Z_y=\eta^{-1}(\{y\})$, consider the maps
\begin{eqnarray*}
\phi,\phi':\quad\ell^p_{[Y]}(Z) & \to & \ell^p(Y)\\
f=\sum_{y\in Y}\lambda_y\mathbf{1}_{Z_y} & \mapsto & \phi(f)=\sum_{y\in Y}\lambda_y\delta_{y} \\
 & \mapsto & \phi'(f)=\sum_{y\in Y}\#(Z_y)^{1/p}\lambda_y\delta_{y}.
\end{eqnarray*}

The following lemma is immediate.
\begin{lem}\label{linid2}
The maps $\phi$ and $\phi'$ are injective 1-Lipschitz linear maps, with dense image.
The mapping $\phi'$ is in addition a bijective isometry. The mapping $\phi$ is surjective (and then a linear isomorphism) if and only if the fibers of $\eta$ have bounded cardinality. If $Z,Y$ are $H$-sets so that $\eta$ is $H$-equivariant, then $\phi$ and $\phi'$ are $H$-equivariant.\qed
\end{lem}

The map $\phi$ seems worse than the map $\phi'$ from the point of view of Lemma \ref{linid2}; however when passing to an affine setting, only the map $\phi$ is workable with. 

Define, for any subset $N\subset Z$, $\ell^p_{[Y],N}(Z)$ as the set of elements in $\ell^p_N(Z)$ (as defined in \S\ref{ala}) that are constant on fibers of $\eta$. In other words, $\ell^p_{[Y],N}(Z)$ is the set of those $f\in\ell^p_{[Y]}(Z)$ such that $f-\mathbf{1}_{N}\in\ell^p(Z)$.

\begin{lem}\label{afid2}
The subset $\ell^p_{Y,N}(Z)$ is nonempty if and only if $N$ coincides up to a finite set to a union of fibers. If so, if $N'$ is the image of $N$ in $Y$, there is an injective affine 1-Lipschitz map, with dense image
\begin{eqnarray*}
\varpi:\quad\ell^p_{[Y],M}(Z) & \to & \ell^p_{M'}(Y)\\
\sum_{y\in Y}\mu_y\mathbf{1}_{Z_y} & \mapsto & \sum_{y\in Y}\mu_y\delta_{y}.
\end{eqnarray*}
The linear part $\ell^p_{[Y]}(Z)\to\ell^p(Y)$ of $\varpi$ is the function $\phi$ of Lemma \ref{linid2}. If $Z,Y$ are $H$-sets so that $\eta$ is $H$-equivariant, then $\varpi$ is $H$-equivariant.\qed
\end{lem}
\begin{proof}
Suppose that $f\in \ell^p_{[Y],N}(Z)$. Define $N'=\{f>1/2\}$. Then $N'$ is a union of fibers and is commensurate to $N$. This proves the first statement. The remainder (except equivariance) follows from the linear counterpart in Lemma \ref{linid2}, and the equivariance statement is immediate.
\end{proof}

\section{$Q$-points on some spaces associated to an $H$-set}

\subsection{Compatible and incompatible points}\label{sett}

Throughout this section, we fix a group (with no topology) $H$ and a homomorphism with dense image $\pi:H\to Q$, where $Q$ is a Hausdorff topological group. We also let $X$ be an $H$-set. 

We freely use the terminology pertaining to $Q$-points introduced in \S\ref{uq}. We compute the $Q$-points successively in various spaces on which $H$ acts isometrically.

\begin{defn}\label{qcom}
If $F$ is a finite subset of $X$, we say that $F$ is {\em $Q$-compatible} if there exists $V\in\mathcal{V}_Q$ such that $gF=F$ for all $g\in V$.

We define the {\em $Q$-compatible part} of $X$ as the union $X^{\langle Q\rangle}$ of $Q$-compatible finite subsets of $X$, and the {\em $Q$-incompatible part} of $X$ as its complement $X^{\rangle Q\langle}=X\smallsetminus X^{\langle Q\rangle}$.
We say that $X$ is $Q$-compatible (resp.\ $Q$-incompatible) if $X=X^{\langle Q\rangle}$ (resp.\ $X=X^{\rangle Q\langle}$).

We say that two elements in $X^{\langle Q\rangle}$ are $Q$-equivalent if they belong to the same $Q$-compatible finite subsets and define the {\em $Q$-reduction} $X^{[Q]}$ as the quotient of $X^{\langle Q\rangle}$ by this equivalence relation, and denote by $\mathcal{P}_Q$ the corresponding partition of $X^{\langle Q\rangle}$.
\end{defn}

Note that the $Q$-equivalence classes, or equivalently the fibers of the canonical projection $X^{\langle Q\rangle}\to X^{[Q]}$, are the minimal $Q$-compatible finite subsets of $X$.

\subsection{$Q$-points in $\mathcal{F}(X)$}

Again, $\pi:H\to Q$ and the $H$-set $X$ are given as in \S\ref{sett}.

Let $\mathcal{F}(X)$ be the set of finite subsets of $X$, with the discrete distance $d(F_1,F_2)=\#(F_1\tu F_2)$; it is a a Boolean algebra for the usual operation. In particular, the $H$-action given by left multiplication on $\mathcal{F}(X)$ is by isometric group automorphisms. Denote by $\rho$ the projection $X^{\langle Q\rangle}\to X^{[Q]}$.

\begin{prop}\label{linearxq}
We have $\mathcal{F}(X)^{Q}=\rho^*\mathcal{F}(X^{[Q]})$, in other words, the set of $Q$-points $\mathcal{F}(X)^{Q}$ consists of the $\rho^{-1}(F)$ for $F$ finite subset of $X^{[Q]}$.
In particular, $\rho$ induces an isomorphism of Boolean algebras \[\mathcal{F}(X)^{Q}\simeq  \mathcal{F}(X^{[Q]}).\] 

Moreover, $X^{\langle Q\rangle}$ is $H$-invariant, the $H$-action factors to a continuous action of $Q$ on $X^{[Q]}$ and the above isomorphism is $Q$-equivariant.
\end{prop}
\begin{proof}
Having Remark \ref{dqdis} in mind, it is immediate that $\mathcal{F}(X)^{Q}$ consists of the $Q$-compatible finite subsets of $X$ as defined in Definition \ref{qcom}. Therefore the first statement immediately follows from the definitions. It is also straightforward that $X^{\langle Q\rangle}$ is $H$-invariant, and the $H$-action factors to a continuous action of $H$ on $X^{[Q]}$ and the isomorphism $\rho^*$ is $H$-equivariant.

Observe that in view of Remark \ref{dqdis}, we have $(X^{[Q]})^Q=X^{[Q]}$, where $X^{[Q]}$ is endowed with the trivial distance (any two distinct elements are at distance 1). By Proposition \ref{cext}, the action of $H$ on $X^{[Q]}$ factors through a continuous action of $Q$. Finally, the identification is $H$-equivariant, hence $Q$-equivariant by density.
\end{proof}

\begin{rem}
It follows from Proposition \ref{linearxq} that if $X$ is $H$-transitive, then either $X$ is $Q$-compatible and $\rho$ is a projection $X\to X^{[Q]}$ with fibers of constant finite cardinal, or $X$ is $Q$-incompatible and $X^{[Q]}=\emptyset$.
\end{rem}

\begin{prop}\label{xqf}
We have 
\[X^{\langle Q\rangle}=\{x\in X:\;\exists V\in\mathcal{V}_Q\;\textnormal{such that}\; V\!x\textnormal{ is finite}\};\]
moreover for every $x\in X^{\langle Q\rangle}$ there exists $V_0\in\mathcal{V}_Q$ such that for any $V\in\mathcal{V}_Q$ contained in $V_0$, the subset $V\!x$ is equal to the fiber $\rho^{-1}(\{\rho(x)\})$.
\end{prop}
\begin{proof}
The inclusion $\subset$ is trivial. Conversely, suppose that $x$ belongs to the right-hand set. Take $V_0$ such that $V_0x$ has minimal cardinality among those $Vx$ with $V\in\mathcal{V}_Q$. Then by minimality, for every $V\in\mathcal{V}_Q$ and $V\subset V_0$ we have $Vx=V_0x$. In particular, for every $W\in\mathcal{V}_Q$, we have $V_0\supset V\cap W\in\mathcal{V}_Q$ and it follows that $V_0x\subset Wx$.

Since $Q$ is a topological group, there exist $V_1,V_2\in\mathcal{V}_Q$ such that $V_1V_2\subset V_0$ and $V_1=V_1^{-1}$. For every $g\in V_1$, we have 
$$gV_0x=gV_2x\in V_1V_2x\subset V_0x;$$
since this also holds for $g^{-1}$, we deduce that $gV_0x=V_0x$. Thus $V_0x\in\mathcal{F}(X)^{Q}$, so that $x\in X^{\langle Q\rangle}$, also proving the additional statement.
\end{proof}

\subsection{$Q$-points in $\mathcal{F}_M(X)$}\label{qfm} We continue with the notation of \S\ref{sett}.

Now let $M$ be a subset of $X$, commensurated by the $H$-action and whose stabilizer is open. Let $\mathcal{F}_M(X)$ be the set of subsets of $X$ commensurate to $M$, with the discrete distance $d(N_1,N_2)=\#(N_1\tu N_2)$. It can be viewed as an affine space over the field on 2 elements $\mathbf{F}_2$, whose linear part is $\mathcal{F}(X)$. It will be sometimes convenient to write the addition in $\mathcal{F}(X)$, the symmetric difference $\tu$, with the sign~$+$.

Define $M^{[Q]}=\rho(M\cap X^{\langle Q\rangle})$. We still denote by $\rho^*$ the inverse image map of $\rho$, from the power set of $X^{[Q]}$ to that of $X^{\langle Q\rangle}$.

\begin{lem}\label{decfm}
If there is an $H$-invariant partition $X=X_1\sqcup X_2$, then, denoting $M_i=M\cap X_i$, 
\[\mathcal{F}_M(X)^{Q}=\mathcal{F}_{M_1}(X_1)^{Q}+\mathcal{F}_{M_2}(X_2)^{Q}.\]
In particular, we have $\mathcal{F}_M(X)^{Q}\neq\emptyset$ if and only if $\mathcal{F}_{M_1}(X_1)^{Q}$ and $\mathcal{F}_{M_2}(X_2)^{Q}$ are both nonempty.
\end{lem}
\begin{proof}
The inclusion $\supset$ is trivial, and conversely $\subset$ follows by decomposing $N$ as $(N\cap X_1)+(N\cap X_2)$. 
\end{proof}

\begin{prop}\label{nodep}
If $N_1,N_2\in \mathcal{F}_M(X)^{Q}$, then $N_1\cap X^{\rangle Q\langle}=N_2\cap X^{\rangle Q\langle}$. In particular, if $\mathcal{F}_M(X)^{Q}\neq\emptyset$, the subset $M^{\rangle Q\langle}$ defined to be equal to $N\cap X^{\rangle Q\langle}$ does not depend on the choice of $N\in\mathcal{F}_M(X)^{Q}$. 

Moreover, $M^{\rangle Q\langle}$ is $H$-invariant, $M^{[Q]}$ is commensurated by the $Q$-action on $X^{[Q]}$ and has an open stabilizer.
\end{prop}

\begin{proof}
We have $N_1\tu N_2\in \mathcal{F}(X)^{Q}$. It follows that $N_1\tu N_2\subset X^{\langle Q\rangle}$, that is, $N_1\cap X^{\rangle Q\langle}=N_2\cap X^{\rangle Q\langle}$.
It follows in particular that $M^{\rangle Q\langle}$ is $H$-invariant. 

Having in mind that $X^{\langle Q\rangle}$ is $H$-invariant and the $H$-action factors to an action of $H$ on $X^{[Q]}$ by Proposition \ref{linearxq}, we obtain that $M^{[Q]}$ is commensurated by the $H$-action. By Proposition \ref{linearxq}, the $H$-action on $X^{[Q]}$ factors through a continuous action of $Q$. Since the action of $Q$ on $X^{[Q]}$ is continuous and $N^{[Q]}$ is commensurate to $M^{[Q]}$ and has an open stabilizer in $Q$, we see that $M^{[Q]}$ has an open stabilizer in $Q$ as well. It then follows, by density, that $M^{[Q]}$ is commensurated by $Q$ as well.
\end{proof}

\begin{prop}\label{fpmx}
Assume that $\mathcal{F}_M(X)^{Q}\neq\emptyset$. We have the equality 
\begin{align*}\mathcal{F}_M(X)^{Q}
& =\rho^*\mathcal{F}_{M^{[Q]}}(X^{[Q]})+\{M^{\rangle Q\langle}\}\\
=&
\Big\{N\in\mathcal{F}_M(X):\;  N\cap X^{\langle Q\rangle}=\rho^{-1}\big(\rho(N\cap X^{\langle Q\rangle})\big)
\textnormal{ and }
N\cap X^{\rangle Q\langle}=M^{\rangle Q\langle}\Big\}.
 \end{align*}
In particular $N\stackrel{\sigma}\mapsto\rho^*(N)+ M^{\rangle Q\langle}$ is a $Q$-equivariant  affine bijection from $\mathcal{F}_{M^{[Q]}}(X^{[Q]})$ to $\mathcal{F}_M(X)^{Q}$. Besides, $\rho^{-1}(M^{[Q]})\smallsetminus M$ is finite.
\end{prop}
\begin{proof}
To prove the formula, by Lemma \ref{decfm}, we can suppose that $X$ is either $Q$-incompatible or $Q$-compatible. The incompatible case is then immediate.
Now assume that $X$ is $Q$-compatible, so the statement to prove can be rewritten as
\[\mathcal{F}_M(X)^{Q}  =\rho^*\mathcal{F}_{M^{[Q]}}(X^{[Q]}) = \big\{N\in\mathcal{F}_M(X):\;  N=\rho^{-1}(\rho(N))\big\}.\]

All these are affine subspaces with the same linear part by Proposition \ref{linearxq}. So it is enough to show that they have a common point. Let us first check that every element in $\mathcal{F}_M(X)^{Q}$ is a union of fibers. Let $N$ be an element therein and let $P$ be a fiber. Then there exists $V\in\mathcal{V}_Q$ such that $gN=N$ and $gP=P$ for all $g\in V$. Thus $g(N\cap P)=N\cap P$, so $N\cap P$ is $Q$-compatible; since $P$ is a fiber, it follows that $N\cap P$ is either empty or equal to $P$, proving that $N$ is a union of fibers. It easily follows that any element of $\mathcal{F}_M(X)^{Q}$ belongs to the two other spaces.

Let us prove that  $\rho^{-1}(M^{[Q]})\smallsetminus M$ is finite. Note that this means that $M\cap X^{\langle Q\rangle}$ coincides, up to a finite set, with a union of fibers. Note that this property does not change if we replace $M$ by a commensurate subset. Hence we deduce it since it is obviously satisfied by $N$ for $N\in\mathcal{F}(M)^{Q}$. 

The $Q$-equivariance follows from $H$-equivariance.
\end{proof}

\subsection{$Q$-points in $\ell^p(X)$}\label{qlpx}

As previously, $X$ is an $H$-set; we forget here $M$, and we fix a real number $p\in [1,\infty\mathclose[$ and we proceed to describe $Q$-points in $\ell^p(X)$.

\begin{prop}\label{linac2}
We have
$$\ell^p(X)^{Q}=\ell^p_{[X^{[Q]}]}(X^{\langle Q\rangle})\simeq\ell^p(X^{[Q]})$$
(the latter isometric $Q$-equivariant isomorphism being described in \S\ref{soma}).
In particular $\ell^p(X)^{Q}=\{0\}$ if and only if $X$ is $Q$-incompatible.
\end{prop}
\begin{proof}
The right-hand isomorphism $\simeq$ (which denotes an isometric linear isomorphism) follows from Lemma \ref{linid2}.

To show the $\supset$ inclusion, by density it is enough to show that for every $y\in X^{[Q]}$, we have $\mathbf{1}_{\rho^{-1}(\{y\})}\in \ell^p(X)^{Q}$; this is clear from Proposition \ref{linearxq}.

Conversely, suppose that $f\in\ell^p(X)^{Q}$. We have to show that $f$ is supported by $X^{\langle Q\rangle}$ and that $f$ is constant on fibers.

Let us first check the latter assertion.
Suppose that $y,z$ are in the same fiber $C$ of $X^{\langle Q\rangle}\to X^{[Q]}$. Let us verify that $f(y)=f(z)$. Otherwise, assume $f(y)\neq f(z)$ and set $\eps=|f(y)-f(z)|$. By Proposition \ref{xqf}, since $Q$ is a group unifilter, there exists $V_0\in \mathcal{V}_Q$ such that for every $V\in\mathcal{V}_Q$ with $V\subset V_0$ we have $Vy=C$. So for every $\mathcal{V}_Q\ni V\subset V_0$, there exists $g\in V$ such that $gy=z$. So $\|gf-f\|_p\ge |f(z)-f(y)|=\eps$ and thus, since this holds for all $V$, we deduce $f\notin\ell^p(X)^{Q}$.

Let us finally check that $f$ is supported by $X^{\langle Q\rangle}$.  Indeed, suppose that $x\notin X^{\langle Q\rangle}$ and assume by contradiction that $f(x)\neq 0$. Set $\eps=|f(x)|$. There exists $V\in\mathcal{V}_Q$ such that $\|gf-f\|_p\le\eps/2$ for every $g\in V$; in particular, $|f(gx)|\ge\eps/2$ for all $g\in V$. By Proposition \ref{xqf}, $Vx$ is infinite, which contradicts that $f\in\ell^p(X)$. 
\end{proof}

\subsection{$Q$-points in $\ell^p_M(X)$}

Consider again $M\subset X$ commensurated by the $H$-action.

\begin{prop}\label{cafa}
We have 
$$\ell^p_M(X)^Q\neq\emptyset\quad\Leftrightarrow\mathcal{F}_M(X)^{Q}\neq\emptyset.$$ 
Suppose now these conditions hold. Then, defining $M^{\rangle Q\langle}$ as in Proposition \ref{nodep}, we have the following equality of nonempty affine subspaces
\[\ell^p_M(X)^{Q}=\ell^p_{M\cap X^{\langle Q\rangle}}(X^{\langle Q\rangle})^{Q}+\mathbf{1}_{M^{\rangle Q\langle}}.\]
\end{prop}
\begin{proof}
In the first statement, the implication $\Leftarrow$ is clear since if $N\in\mathcal{F}_M(X)^{Q}$ then $\mathbf{1}_N\in\ell^p_M(X)^{Q}$.

Conversely, assume that $\ell^p_M(X)^Q\neq\emptyset$, namely contains an element $f$. Then $f^{-1}([1/3,2/3])$ is finite. In particular, there exists $\alpha$ and $\eps>0$ such that $[\alpha-\eps,\alpha+\eps]\subset \mathopen]0,1\mathclose[$ and $[\alpha-\eps,\alpha+\eps]\cap f(X)=\emptyset$. 
Then there exists $V\in\mathcal{V}_Q$ such that $\sup_{g\in V}\|f-gf\|\le \eps$; in particular $\|f-gf\|_\infty\le\eps$. Define $N=f^{-1}([\alpha,+\infty\mathclose[)$. If $x\in X$ and $g\in V$, then either $f(x)>\alpha+\eps$, so $f(x)>\alpha$ and both $x$ and $gx$ belong to $N$, or $f(x)<\alpha-\eps$ and then similarly both $x$ and $gx$ belong to $N$. In particular, $gN=N$ for all $x\in V$ and hence $N\in\mathcal{F}(X)^{Q}$.

The second statement is clear.
\end{proof}

Under the assumption that $\mathcal{F}_M(X)^{Q}\neq\emptyset$, this reduces the study to the case where $X$ is $Q$-compatible. Using the notation of \S\ref{soma}, we have:

\begin{cor}\label{lpmx}
Assume that $X$ is $Q$-compatible and $\mathcal{F}_M(X)^{Q}\neq\emptyset$. Then
\[\ell^p_M(X)^{Q}=\ell^p_{[X^{[Q]}],M}(X).\]
If moreover $X$ is $H$-transitive, then the latter is $Q$-equivariantly isometric to $\ell^p_{M^{[Q]}}(X^{[Q]})$.
\end{cor}
\begin{proof}
If $N\in\mathcal{F}_M(X)^{Q}$, then both terms contain $\mathbf{1}_N$, and so both are affine subspaces. They have the same linear part $\ell^p(X)^{Q}=\ell^p_{[X^{[Q]}]}(X)$ by Proposition \ref{linac2}, so the equality follows. For the last statement, use the map $\phi$ of \S\ref{soma}: the transitivity implies that the projection $X\to X^{[Q]}$ has fibers of constant cardinal, hence $\phi'$ is a homothety.
\end{proof}

\section{Applications of $Q$-points}

\subsection{Existence of $Q$-points}

\begin{thm}\label{eqlp}
Let $H$ be a group with a homomorphism with dense image into a Hausdorff topological group $Q$. Let $X$ be an $H$-set and $M$ a commensurated subset, and define the associated function
\[\elk_M(g)=\#(M\tu gM),\quad g\in H.\] We have equivalences, for a given $1\le p<\infty$:
\begin{enumerate}
\item\label{rs1} $\ell^p_M(X)^Q\neq\emptyset$; 
\item\label{rs2} $\mathcal{F}_M(X)^Q\neq\emptyset$;
\item\label{rs3} there exists a neighborhood $V$ of 1 in $Q$ such that $\elk_M$ is bounded on $\pi^{-1}(V)$;
\item\label{rs4} there exist an open subgroup $\Omega$ of $Q$ and a subset $N$ commensurate to $M$ that is $\pi^{-1}(\Omega)$-invariant.
\end{enumerate}
\end{thm}
In spite of the short proof below, we should emphasize that the tricky implications, for which the work was done beforehand, are $(\ref{rs1})_p\Rightarrow$(\ref{rs2}) and (\ref{rs3})$\Rightarrow (\ref{rs1})_2$.

\begin{proof}
Denote by $(\ref{rs1})_p$ the Property $(\ref{rs1})$ for a given $p$.
To prove the equivalences, we are going to prove 

\[(\ref{rs1})_p\Rightarrow(\ref{rs2})\Rightarrow(\ref{rs4})\Rightarrow(\ref{rs3})\Rightarrow(\ref{rs1})_2 \textnormal{ and } (\ref{rs2})\Rightarrow(\ref{rs1})_p.\]

$(\ref{rs1})_p\Rightarrow$(\ref{rs2}) is part of Proposition \ref{cafa}.

(\ref{rs2})$\Rightarrow$(\ref{rs4})
Choose $N\in\mathcal{F}_M(X)^Q$. By Proposition \ref{cext}, the function $f:Q\to\mathbf{R}$, $g\mapsto\|\mathbf{1}_N-g\mathbf{1}_N\|_p^p$ is well-defined and continuous. Since it takes integer values, the set $\Omega=\{g:f(g)=0\}$ is open; it is clearly a subgroup.

(\ref{rs4})$\Rightarrow$(\ref{rs3}) is trivial.

(\ref{rs3})$\Rightarrow (\ref{rs1})_2$
Suppose that $\elk_M$ is bounded on $\pi^{-1}(V)$. By Theorem \ref{coexx} and using that a Hilbert space is a complete CAT(0) space, we deduce $\ell^2_M(X)^Q\neq\emptyset$. 

(\ref{rs2})$\Rightarrow(\ref{rs1})_p$ is clear since if $N\in\mathcal{F}_M(X)^Q$, then $\mathbf{1}_N\in\ell^p_M(X)^Q$.
\end{proof}

Note that by specifying the equivalence (\ref{rs3})$\Leftrightarrow$(\ref{rs4}) to $Q=1$, we get the classical result of Brailovsky, Pasechnik and Praeger \cite{BPP}:

\begin{cor}\label{cabou}
Let $H$ be a group acting on a set $X$ and $M$ a commensurated subset. If $\sup_{g\in H}\#(M\tu gM)<\infty$, then there exists a subset $N$ commensurate to $M$ and $H$-invariant.\qed
\end{cor}

\begin{thm}\label{fwres}
Let $\pi:H\to Q$ be a homomorphism with dense image between locally compact groups. Suppose that $\pi$ is a Kazhdan homomorphism (\S\ref{kh}).  Then $H$ has Property FW (resp.~FW') if and only if $Q$ has Property FW (resp.~FW').
\end{thm}
\begin{proof}
Clearly if $H$ has Property FW or FW' then so does $Q$. Assume that $Q$ has Property FW. Let $X$ be a continuous discrete $H$-set and $M$ a commensurated subset. 

By Theorem \ref{afr}, for every continuous affine isometric action of $H$ on a Hilbert space $V$, we have $V^Q\neq\emptyset$. We apply this to $\ell^2_M(X)$, so $\ell^2_M(X)^Q\neq\emptyset$. By Proposition \ref{cafa}, we get $\mathcal{F}_M(X)^Q\neq\emptyset$. By Proposition \ref{nodep}, we deduce that $M^{[Q]}$ is commensurated by the $Q$-action, has an open stabilizer in $Q$, and $M\cap X^{\rangle Q\langle}$ is transfixed.

If $Q$ has Property FW, it follows that $M^{[Q]}$ is transfixed and its inverse image in $M^{\langle Q\rangle}$ is commensurate to $M\cap X^{\langle Q\rangle}$; it follows that the latter is transfixed. So $M$ is transfixed.

If $Q$ has only Property FW' and the $H$-set $X$ is assumed transitive, we obtain the same conclusion.
\end{proof}

\begin{proof}[Proof of Theorem \ref{ipart} for Property FW]
Write $S=S_1\times S'$ where $S_1$ has Property T, so that the projection of $\Gamma$ on $S'$ has a dense image. By Theorem \ref{mbl}, the projection $\Gamma\to S'$ is a Kazhdan homomorphism. Since $S'$ is virtually connected, it obviously has Property FW. Hence by Theorem \ref{fwres}, $\Gamma$ has Property FW. 
\end{proof}

\subsection{Application to Property FW and group extensions}

If $G$ is a locally compact group and $H$ a subgroup, we say that $(G,H)$ has relative Property FW if for every continuous action of $G$ on a discrete set $X$ commensurating a subset $M$, the subgroup $H$ leaves invariant some subset commensurate to $M$. Recall that $\elk_M$ is defined by $\elk_M(g)=\#(M\tu gM)$.

\begin{lem}\label{xtfw}
Let $G$ be a locally compact group, and $N_1\subset N_2$ closed normal subgroups. Suppose that $(G/N_1,N_2/N_1)$ has relative Property FW. Suppose that $X$ is a continuous discrete $G$-set and $M$ a commensurated subset. Suppose that $\elk_M$ is bounded on $N_1$. Then $\elk_M$ is bounded on $N_2$.
\end{lem}
\begin{proof}
By assumption, some subset $N$ commensurate to $M$ is $N_1$-invariant. So $\mathcal{F}_M(X)^{G/N_1}\neq\emptyset$. By Proposition \ref{nodep}, $M^{[G/N_1]}$ is commensurated by the $G/N_1$-action on $X^{[G/N_1]}$. So it is commensurate to an $N_2/N_1$-invariant subset, by Property FW of $(G/N_1,N_2/N_1)$. Pulling back to $X$ and taking the union with $M^{\rangle N_2/N_1\langle}$ (see the notation in Proposition \ref{nodep}), we obtain a subset commensurate to $M$ and $N_2$-invariant.
\end{proof}

\begin{cor}\label{extfw}
Suppose that $G$ is a locally compact, compactly generated group.
Suppose that $N_1\subset N_2$ are closed normal subgroups of $G$, and that $(G,N_1)$ and $(G/N_2,N_1/N_2)$ have relative Property FW. Then $(G,N_2)$ has relative Property FW.\qed
\end{cor}

\begin{cor}\label{extfww}
Properties FW is closed under taking group extensions.
\end{cor}

This can also be proved (less directly) using the following characterization: a locally compact group has Property FW if and only if every continuous action on a nonempty oriented connected median graph has a fixed vertex. 

\section{Property FM}\label{s_ame}

\subsection{Eymard-amenability and property FM}

Let $G$ be a topological group and $X$ a continuous discrete $G$-set. Recall that the $G$-set $X$ is called {\em Eymard-amenable} if there is a $G$-invariant mean on $X$. Note that this trivially holds if there is a finite orbit in $X$.

\begin{defn}
We say that the topological group $G$ has
\begin{itemize}
\item {\em Property FM} if every continuous Eymard-amenable continuous discrete $G$-set has a finite orbit;
\item {\em Property FM'} if every continuous {\em transitive} Eymard-amenable continuous discrete $G$-set is finite.
\end{itemize}
\end{defn}

Note that FM trivially implies FM'. The converse is not true for countable discrete groups, see Corollary \ref{slnq}.

\begin{que}
Are Properties FM and FM' equivalent for finitely generated groups?
\end{que}

Property FM' is usually easier to check, but less convenient to deal with formally.

The following fact is immediate.

\begin{fact}\label{squot}
Let $G,H$ are topological groups and $f:G\to H$ is a continuous homomorphism with dense image. If $G$ has Property FM, so does $H$.
\end{fact}

\begin{prop}\label{fmext}
Properties FM and FM' are closed under taking extensions of topological groups: if $G$ is a topological group, $N$ a normal subgroup, and if both $N$ and $Q=G/N$ have Property FM (resp.\ FM'), then so does $G$.
\end{prop}
\begin{proof}
Suppose that $N$ and $Q$ have Property FM. Let $X$ be a continuous discrete $G$-set with an invariant mean. By Property FM of $N$, this mean is supported by the union $Y$ of finite $N$-orbits in $X$. Let $Y'$ be the quotient of $Y$ by the $N$-action. Then $Y'$ is a discrete $Q$-set with an invariant mean, and the $Q$-action on $Y'$ is continuous: indeed if $y'\in Y$, let $y$ be a preimage of $y'$ in $y$; the stabilizer $G_y$ is open, and hence the image of $G_y$ in $Q$ is open, and is contained in $Q_{y'}$, so $Q_{y'}$ is open. Accordingly, by Property FM of $Q$, there is a finite $Q$-orbit in $Y$. Its inverse image is thus a finite $G$-orbit in $X$.

Now suppose that $N$ and $Q$ have Property FM'. Let $X$ be a transitive continuous discrete $G$-set with an invariant mean. Let $X'$ be the quotient of $X$ by the $N$-action; again it is a continuous discrete $Q$-set and has a $Q$-invariant mean, and therefore by Property FM' for $Q$, we deduce that $X'$ is finite. So the $N$-orbits form a finite partition of $X$; hence each of the $N$-orbits has the same nonzero mean; by Property FM' for $N$, these orbits are finite and hence $X$ is finite.
\end{proof}

\begin{rem}
If $G$ is an arbitrary topological group (not necessarily Hausdorff), Proposition \ref{fmext} can be applied when $N$ is the closure of $\{1\}$. This shows that $G$ has Property FM if and only if the Hausdorff quotient $G/\overline{\{1\}}$ has Property FM.
\end{rem}

The following proposition was obtained, with a similar proof, in \cite[Lemma 2.14]{GM} for countable discrete groups. The latter was itself was inspired by Kazhdan's proof that countable discrete groups with Property T are finitely generated. Recall that if $G$ is a group and $H$ is a subgroup, $G$ is {\em finitely generated over $H$} if $G$ is generated by the union of $H$ and some finite subset of $G$.

\begin{prop}\label{fmcg}
Let $G$ be a topological group with Property FM. Then it is finitely generated over any open subgroup. In particular, if $G$ is locally compact then it is compactly generated.
\end{prop}
\begin{proof}
Let $H$ be an open subgroup and let $(G_i)$ be the family of open subgroups finitely generated over $H$. Then $I$ is a net (for the inclusion of the $G_i$). The disjoint union of the discrete $G$-sets $G/G_i$ is Eymard-amenable: indeed, if $x_i$ is the base-point in $G/G_i$ and if $m$ is a mean obtained as a limit point of the Dirac measures at $x_i$ (thus $m$ is actually an ultrafilter, i.e.\ takes values in $\{0,1\}$), then $m$ is $G$-invariant. By Property FM, there exists a finite $G$-orbit. In other words, some $G_i$ has finite index in $G$. Since $G_i$ is finitely generated over $H$, it follows that $G$ is also finitely generated over $H$.

As for the second statement: if $G$ is locally compact, then it always possesses a compactly generated open subgroup $H$ (namely, the subgroup generated by any compact neighborhood of 1). By the first statement, $G$ is finitely generated over $H$; it follows that $G$ is compactly generated.
\end{proof}

\begin{rem}\label{fwsb}
If $G$ is an uncountable discrete group, then it does not have Property FM by Proposition \ref{fmcg}. On the other hand, many such groups, including permutation groups of infinite sets, are {\em strongly bounded} in the sense that every isometric action on any metric space has bounded orbits \cite{Ber}; this implies Property FW for such groups.
\end{rem}

\begin{prop}\label{tfw}
Let $G$ be a locally compact group with Property T. Then it has Property FM.
\end{prop}

\begin{lem}[Eymard {\cite[Expos\'e n\textsuperscript{o}3]{Eym}}]\label{eyml2}
Let $G$ be a locally compact group and $X$ a continuous discrete $G$-set. Then $\ell^2(X)$ has nonzero invariant vectors if and only if $X$ has a finite $G$-orbit, and has almost invariant vectors if and only if the $G$-set $X$ is Eymard-amenable.
\end{lem}
\begin{proof}[On the proof] The first statement is trivial and the $\Rightarrow$ implication of the second statement is an easy argument. A standard argument, using the weak-* topology, also allows to prove the $\Leftarrow$ implication when $G$ is discrete, but in general, it only shows that if $X$ is Eymard-amenable, then $\ell^2(X)$ has almost invariant vectors as a representation of the underlying discrete group. A significant amount of additional work (essentially following an argument of Reiter) is needed. We refer to Eymard \cite[Expos\'e n\textsuperscript{o}3]{Eym}.
\end{proof}

\begin{proof}[Proof of Proposition \ref{tfw}]
Let $G$ have Property T. Let $X$ be a continuous discrete $G$-set with an invariant mean. Then $\ell^2(X)$ almost has invariant vectors, by Lemma \ref{eyml2}. By Property T, $\ell^2(X)$ has nonzero invariant vectors, so $X$ has a finite $G$-orbit.
\end{proof}

Let us also mention the well-known

\begin{lem}\label{multia}
Let $G$ be a locally compact group and a family $(H_i)$ of amenable open subgroups. If the $G$-set $X=\bigsqcup_i G/H_i$ is Eymard-amenable, then $G$ is amenable.
\end{lem}
\begin{proof}
Let $G$ act on a convex compact subset $K$ of a locally convex vector space, by affine transformations. Then for each $i$, there exists a point $x_i\in K$ fixed by $H_i$. So there is a $G$-equivariant function $X\to K$, mapping the base-point of $G/H_i$ to $x_i$. The push-forward of an invariant mean on $\ell^\infty(X)$ provides an invariant mean on the space $\ell^\infty(K)$. The latter thus restricts to a probability measure, given by a linear form on $C(X)$. The barycenter of this probability measure (see \cite[Theorem 2.29]{Luk}) is a fixed point by~$G$.
\end{proof}

\begin{prop}\label{fmfi}
Let $G$ be a topological group and $H$ an open subgroup of finite index. Then $G$ has Property FM (resp.\ FM') if and only if $H$ has Property FM (resp.\ FM'). Actually, the implications $\Leftarrow$ also hold when $H$ is closed cocompact in $G$.
\end{prop}
\begin{proof}
Suppose that $H$ is cocompact and has Property FM. Let $G$ act on $X$ preserving a mean. Then $H$ has a finite orbit, so $G$ as well. 

For the same implication with Property FM', it is convenient to first observe that a topological group $G$ has Property FM' if and only if for every continuous discrete $G$-set with {\em finitely many} orbits and an invariant mean, there is a finite orbit. Now observe that every orbit $G/\Omega$ splits into finitely many $H$-orbits, because $H\backslash (G/\Omega)=(H\backslash G)/\Omega$ is finite by a compactness argument. Since, for a $G$-action on a discrete set, having finitely many orbits is preserved by restricting to $H$, the same argument as for Property FM thus carries over.

Now assume that $H$ has finite index and $G$ has Property FM'. Let $H$ act transitively, say on $H/L$ with $L$ open, with an invariant mean. Then given the action of $G$ on $G/L$, the same mean is preserved by $H$, so by averaging, we obtain a $G$-invariant mean on $G/L$. By Property FM', it follows that $G/L$ is finite and hence $H/L$ is finite, so $H$ has Property FM'.

Finally assume that $G$ has Property FM. Any $H$-action can be described as the action on the disjoint union $\bigsqcup H/L_i$ for a suitable family of open subgroups $L_i$; we assume it has an invariant mean. The latter set sits inside the disjoint union $\bigsqcup G/L_i$, with the same mean preserved by $H$, so by averaging we obtain a $G$-invariant mean, and by Property FM there is a finite orbit $G/L_i$, so $H/L_i$ is finite as well. 
\end{proof}

Let us say that a topological group is {\em aperiodic} if it has no proper open subgroup of finite index, and is {\em virtually aperiodic} if it has a finite index open aperiodic subgroup, or equivalently if it has only finitely many finite index open subgroups.

Let us also mention the following result extracted from \cite{GM} about free products. 

\begin{thm}[Glasner and Monod]\label{glamo}
Let $H_1,H_2$ be nontrivial countable discrete groups and consider their free product $G=H_1\ast H_2$. Then we have the following:
\begin{enumerate}
\item\label{oth} If both $H_1$ and $H_2$ have a finite proper quotient, then $H_1\ast H_2$ does not have Property FM' (and hence does not have Property FM)
\item\label{oth2} If $H_1$ is infinite and $H_2$ is not virtually aperiodic, then $H_1\ast H_2$ does not have Property FM' (and hence does not have Property FM);
\item\label{vafm} Suppose that none of the previous two applies, i.e.\ if $H_1$ and $H_2$ are virtually aperiodic and either $H_1$ or $H_2$ is aperiodic. Then if both $H_1$ and $H_2$ have Property FM, so does $H_1\ast H_2$.
\end{enumerate}
In particular, $H_1\ast H_2$ has Property FM if and only if $H_1$ and $H_2$ have Property FM and are virtually aperiodic, and at least one is aperiodic.
\end{thm}
\begin{proof}[On the proof]
The statement as given in \cite{GM} is formulated with a certain property $\mathcal{F}$ which means ``virtually (aperiodic FM)"; in view of Proposition \ref{fmfi} it also means ``(virtually aperiodic) and FM"; in the second case the free product is infinite and they additionally produce a faithful action, forcing the set to be infinite.

We easily see (\ref{oth}): if both $H_1$ and $H_2$ have a nontrivial finite quotient, then $H_1\ast H_2$ has a free product of two nontrivial finite groups as a quotient and hence has a infinite virtually abelian quotient, discarding Property FM. 

Let us now justify (\ref{oth2}). Up to exchange $H_1$ and $H_2$, we can assume that $H_1$ is aperiodic and $H_2$ is not virtually aperiodic. Then (following \cite{GM}, without bothering with faithfulness), we let $(X_n)$ be a sequence of finite $H_2$-sets of increasing cardinal; consider the infinite set $X=\bigsqcup X_n$ with the natural action of $H_2$; consider a subset $Y\subset X$ intersecting $X_n$ in a singleton for each $n$, fix a transitive action of $H_1$ on $Y$ and extend it to an action of $H_1$ on $X$ by acting trivially elsewhere. This defines a transitive action of $H_1\ast H_2$ on $X$ for which the normalized constant function $f_n$ on $X_n$ is a sequence of almost invariant vectors, so that $X$ is Eymard-amenable.

Finally let us mention a proof of (\ref{vafm}). First assume that $H_1$ and $H_2$ are both aperiodic with Property FM. Then if $G$ acts on a set with an invariant mean, it follows that the mean is both supported by $H_1$-fixed points and by $H_2$-fixed points, hence by $G$-fixed points; thus there is an orbit reduced to a singleton. In general, we can suppose $H_1$ aperiodic and $H_2$ has a minimal finite index subgroup $M$, which has Property FM by Proposition \ref{fmfi}. Then the kernel of $G\to H_2/M$ is isomorphic to $M\ast H_1^{\ast G/M}$, which is a free product of $[G:M]+1$ aperiodic groups with Property FM and thus has Property FM by the previous case. So by the reverse direction of Proposition \ref{fmfi}, $G$ has Property FM as well.
\end{proof}

Let us also mention a variant of \cite[Lemma 4.5]{GM}.

\begin{prop}\label{tars}
Let $G$ be a non-amenable locally compact group in which every open subgroup of infinite index is amenable. Then $G$ has Property FM.
\end{prop}
\begin{proof}
Let $X$ be a continuous discrete Eymard-amenable $G$-set. If there is no finite orbit, then we can apply Lemma \ref{multia} to deduce that $G$ is amenable, a contradiction. So there is a finite orbit, proving that $G$ has Property FM.
\end{proof}

\begin{que}\label{fgwt}
Does there exist a finitely generated non-amenable group without Property T, in which every proper subgroup is finite?
\end{que}

I expect a positive answer to Question \ref{fgwt}, which would provide an application of Proposition \ref{itars} in the discrete setting. Note that I do not require the group to have finite exponent in Question \ref{fgwt}.

\begin{rem}
There are no known examples of finitely generated groups satisfying the hypotheses of Proposition \ref{tars} and known not to have Kazhdan's Property T. On the other hand, there are many natural instances in the locally compact setting, such as $\SL_2(\mathbf{Q}_p)$ or the automorphism group of a regular tree of finite valency $\ge 3$.
\end{rem}

\begin{rem}\label{FMtau}
Say that an orthogonal representation of a locally compact group $G$ has a spectral gap if the orthogonal of the subspace of invariant vectors does not almost have invariant vectors. Property FM can be stated as: the representations $\ell^2(X)$ for $X$ continuous discrete $G$-set without finite orbits have a spectral gap. 
Bekka and Olivier \cite{BO} consider the following stronger property (*): the representations $\ell^2(X)$ for $X$ continuous discrete $G$-set have a spectral gap. Thus, as observed in \cite[Remark 16]{BO}, Property (*) is equivalent to the conjunction of Property (FM) and the well-known Property ($\tau$), which is defined as: $G$ has Property ($\tau$) if the $G$-representations $\ell^2(X)$, when $X$ range over continuous discrete $G$-sets with only finite orbits, have a spectral gap. Bekka and Olivier characterize Property (*) as an $\ell^p$-analogue of Property T, where $p$ is an arbitrary number in $\mathopen]1,\infty\mathclose[\smallsetminus\{2\}$.
\end{rem}

\section{Hilbertian methods for Property FM}\label{hilbertian}

\subsection{Reduced 1-cohomology}\label{red1c}

Let $G$ be a locally compact group, $\pi$ an orthogonal representation in a Hilbert space $\mathcal{H}_\pi$. Recall that a 1-cocycle is a continuous function $b:G\to \mathcal{H}_\pi$ such that $b(gh)=\pi(g)b(h)+b(g)$ for all $g,h\in G$. A 1-coboundary is a function of the form $g\mapsto\xi-\pi(g)\xi$; this is always a 1-cocycle. The space of $Z^1(G,\pi)$ of 1-cocycles is a Frechet space under the topology of uniform convergence on compact subsets. The subspace of 1-coboundaries is denoted by $B^1(G,\pi)$, and its closure $\overline{B^1}(G,\pi)$ consists by definition of {\em almost 1-coboundaries}. We also say that 1-cocycles $b,b'$ are {\em cohomologous} (resp.\ {\em almost cohomologous}) if $b-b'\in B^1(G,\pi)$ (resp.\ $b-b'\in\overline{B^1}(G,\pi)$). The quotient $Z^1(G,\pi)/\overline{B^1}(G,\pi)$ is denoted by $\overline{H^1}(G,\pi)$ and is called reduced first cohomology space of $\pi$.

Recall also that $\pi$ {\em almost has invariant vectors} if for every compact subset $K$ of $G$ and $\eps>0$ there exists $\xi\in\mathcal{H}_\pi$ of norm 1 such that $\|\xi-\pi(g)\xi\|\le\eps$. 

\begin{lem}[Guichardet, {\cite[Proposition 2.12.2]{BHV}}]
If the orthogonal of the invariant vectors in $\pi$ does {\em not} have almost invariant vectors then the subspace $B^1(G,\pi)$ is closed; the converse holds if $G$ is $\sigma$-compact.\qed
\end{lem}

Let us introduce another property, related to Property dFH from Definition \ref{ddfh}.

\begin{defn}
We say that a locally compact group $G$ has Property {\odFH} if for every continuous discrete $G$-set $X$, we have $\overline{H^1}(G,\ell^2(X))=0$.
\end{defn}

This property can be checked on transitive $G$-sets $X$, because the class of unitary representations with vanishing reduced first cohomology is stable under taking (possibly infinite) orthogonal direct sums (as a particular case of \cite[Lemma 3.2.4]{BHV}). Note that for a discrete group $\Gamma$, Property {\odFH} implies the vanishing of the first $\ell^2$-Betti number, which means by definition that $\overline{H^1}(\Gamma,\ell^2(\Gamma))=0$.

\begin{prop}\label{dfhodfh}
A locally compact, $\sigma$-compact group $G$ has Property dFH if and only it has all three properties \odFH, FM, and $\tau$.
\end{prop}
\begin{proof}
If $G$ fails to have Property FM or $\tau$, then it admits a continuous discrete $G$-set $X$ such that $\ell^2(X)$ has a spectral gap (see Remark \ref{FMtau}.  Since $G$ is $\sigma$-compact, Guichardet's lemma implies that $H^1(G,\ell^2(X))\neq 0$, hence Property dFH fails. Thus the ``only if" implication holds.

Conversely if $G$ has all three properties, let $X$ be a discrete continuous $G$-set. By Property \odFH, $\bar{H^1}(G,\ell^2(X))=0$. On the other hand, $\ell^2(X)$ has a spectral gap (see Remark \ref{FMtau}) and by Guichardet's lemma, this implies that $H^1(G,\ell^2(X))$ is Hausdorff. Hence $H^1(G,\ell^2(X))=0$.
\end{proof}

Let us introduce the transitive version of Property dFH (for \odFH, we have seen that it is unnecessary).

\begin{defn}
We say that a locally compact group $G$ has Property dFH' if for every continuous discrete transitive $G$-set $X$, we have $H^1(G,\ell^2(X))=0$.
\end{defn}

\begin{prop}
Let $G$ be a locally compact group with Property dFH. Then it has Property FW; if moreover $G$ is $\sigma$-compact, then $G$ has Property FM and $\tau$. The same holds with the transitive versions dFH', FW', FM'.
\end{prop}
\begin{proof}
If $G$ fails to have Property FW, then consider a continuous discrete $G$-set $X$ and a commensurated subset $M$ that is not transfixed. Then the usual cocycle $g\mapsto\mathbf{1}_M-g\cdot\mathbf{1}_M$ is unbounded in $\ell^2(X)$, hence $G$ fails to have Property dFH.

The case of Property FM is contained in Proposition \ref{dfhodfh}, and the corresponding proofs in the transitive setting (dFH' implies FW' and, under $\sigma$-compactness, implies FM') are exactly the same.
\end{proof}

\begin{prop}\label{pnarf}
Let $G$ be a compactly generated locally compact group with Property \odFH. If $X$ is a continuous discrete $G$-set with a non-transfixed commensurated subset $M$, then some infinite $G$-orbit $Y$ in $X$, such that $Y\cap M$ is non-transfixed, is Eymard-amenable. In particular, if $G$ has Property FM then it has Property FW.
\end{prop}

\begin{proof}
Since $G$ is compactly generated, by \cite[Proposition 4.7]{CorFW}, there exists an infinite $G$-orbit $Y\subset X$ such that $Y\cap M$ is not transfixed. Let $b$ be the cocycle $b(g)=\mathbf{1}_M-\mathbf{1}_{gM}$, which is by assumption unbounded (in view of Corollary \ref{cabou}). If by contradiction $Y$ is not Eymard-amenable, then the representation of $G$ on $\ell^2(Y)$ does not almost have invariant vectors (Lemma \ref{eyml2}). By Guichardet's lemma, it follows that $H^1(G,\ell^2(X))$ is Hausdorff. On the other hand, by Proposition \ref{narf0}, $\overline{H^1}(G,\ell^2(X))=0$. Hence $H^1(G,\ell^2(X))=0$, contradicting the unboundedness of $b$.
\end{proof}

\subsection{Superrigidity of square-integrable lattices}

We need to recall the notion of {\em square-integrable} lattice. Let $G$ be a compactly generated locally compact group (endowed with a left Haar measure) and $\Gamma$ a lattice (thus $\Gamma$ is countable). Choose a measurable fundamental domain $D$ of finite measure, so that $G$ is the disjoint union $\bigsqcup_{\gamma\in\Gamma}D\gamma$. Let $\alpha_D:G\times D\to\Gamma$ be the associated cocycle, defined by: $\alpha_D(g,x)=\gamma$ if $gx\gamma\in D$; then $\alpha_D$ is measurable. 

\begin{defn}
The lattice $\Gamma$ is called {\em square-integrable} if it is finitely generated, and for some choice $D$ of measurable fundamental domain $D$, we have:
denoting by $\ell$ the word length on $\Gamma$ with respect to some finite generating subset, we have
\[\int_D\ell(\alpha_D(g,x))^2dx<\infty,\quad\forall g\in G.\]
(This does not depend on the choice of $\ell$.)
\end{defn}

\begin{exe}\label{exint}~
\begin{enumerate}
\item Clearly, cocompact lattices are square-integrable since $D$ can be chosen to be bounded.
\item\label{exint2} All lattices in simple Lie groups of rank 1, finite center and\footnote{(Addendum April 2020) In the published version, the square-integrability of lattices in rank~1 simple Lie groups $G$ was asserted, without excluding the case when $\dim(G)\le 6$, that is, $G$ locally isomorphic to $\mathrm{SO}(2,1)$ (3-dimensional) or $\mathrm{SO}(3,1)$ (6-dimensional). But in 3-dimensional simple Lie groups, non-cocompact lattices are not square-integrable, as follows from Lemma 5.4 in: [U.\ Bader, A.\ Furman, R.\ Sauer, {\it Integrable measure equivalence and rigidity of hyperbolic lattices.} Invent. Math. 194 (2013), no. 2, 313--379]. The misquotation was due to the ambiguous hypothesis {\it ``any other rank one simple Lie group"} of \cite[Th.\ 3.7]{Shann}, but indeed checking the proof in \cite{Shann} all cases are checked except precisely that of simple Lie groups locally isomorphic to $\mathrm{SO}(2,1)$ and $\mathrm{SO}(3,1)$. I thank U.\ Bader for pointing out the mistake and the above reference. The mistake has no consequence in this paper since this very result is only stated for information purposes (unlike Example \ref{exint}(\ref{exint4}), which is implicitly used when saying, before stating Corollary \ref{cowm} below, that Theorem \ref{iccj} follows).} dimension $>6$ are square-integrable \cite[Theorems 3.6 and 3.7]{Shann}.
\item ``Twin building Kac-Moody" lattices are square-integrable \cite[Theorem 31]{CR}.
\item\label{exint4} Let $n\ge 2$ and $G=G_1\dots G_n$ be a semisimple group of algebraic type (as in Definition \ref{algty}), and let $\Gamma$ be an irreducible lattice, in the sense that $G_i\Gamma$ is dense for all $i$ (i.e.\ the projection of $\Gamma$ in $G/G_i$ is dense for all $i$). Then $\Gamma$ is square-integrable: this essentially follows from results of Margulis \cite[Chap. VIII, Proposition 1.2]{Mar} and the non-distortion of these lattices, due to Lubotzky, Mozes and Raghunathan \cite{LMR}; see the discussion in \cite[\S 2]{Sha}.
\end{enumerate}
\end{exe}

The following is a restatement of Shalom's superrigidity of reduced 1-cohom\-ology theorem \cite[Theorem 4.1]{Sha} using the convenient notion of $Q$-points from \S\ref{uq} (here $Q=G_i$ and they will be called $G_i$-vectors since the action is linear). 

\begin{thm}[Shalom]\label{sha}
Let $G$ be a compactly generated locally compact group given as a product $G=G_1\times\dots \times G_n$ (where $n\ge 2$) and let $\Gamma$ be an irreducible square-integrable lattice.

Let $\pi$ be an orthogonal Hilbertian representation of $\Gamma$ and $b\in\mathbf{Z}^1(\Gamma,\pi)$ a cocycle. Let $\pi_i$ be the subrepresentation of $G_i$-vectors. Then there exist $b_i\in Z^1(G_i,\pi_i)$ ($i=1,\dots,n$) such that, writing $b'_i=b_i|_\Gamma$, the cocycle $b$ is almost cohomologous to $\sum b'_i$. In particular, if $\pi$ does not have almost $\Gamma$-invariant vectors, then $b$ and $b'$ are cohomologous.
\end{thm}

\begin{prop}\label{narf0}
Let $G$ be a compactly generated locally compact group given as a topological almost direct product (in the sense of Definition \ref{algty}) $G=G_1\times\dots \times G_n$ (where $n\ge 2$) and let $\Gamma$ be an irreducible square-integrable lattice. If each $G_i$ has Property \odFH, then so does $\Gamma$.
\end{prop}

\begin{proof}
Pulling back if necessary, we can suppose that the product is direct.

Let $\pi$ be a unitary representation of $\Gamma$ and consider $b\in Z^1(\Gamma,\pi)$. By Theorem \ref{sha}, $b$ is almost cohomologous to a cocycle of the form $\sum b'_i$, with $b'_i=b_i|_\Gamma$ and $b_i\in Z^1(G_i,\ell^2(X)^{G_i})$. By Proposition \ref{linac2}, we have $H^1(G_i,\ell^2(X^{[G_i]}))=H^1(G_i,\ell^2(X)^{G_i})$. By Property {\odFH} for $G_i$, we obtain that $b_i$ is almost cohomologous to zero (as a cocycle on $G_i$), and hence $b'_i$ is almost cohomologous to zero as well, and thus $b$ is cohomologous to zero.
\end{proof}

Using Proposition \ref{pnarf}, we deduce the following corollary, which concludes the proof of Theorem \ref{iccj}.

\begin{cor}\label{cowm}
If $\Gamma$ is as in Proposition \ref{narf0}, then Property FM for $\Gamma$ implies Property FW for $\Gamma$.\qed
\end{cor}

Here is now the analogue of Theorem \ref{fwres} for Properties FM and dFH.

\begin{thm}\label{fmres}
Let $\pi:H\to Q$ be a homomorphism with dense image between locally compact groups. Suppose that $\pi$ is a Kazhdan homomorphism (\S\ref{kh}). Let $(P)$ be one of the properties: dFH, \odFH, FM, FM'. Then $H$ has Property (P) if and only $Q$ has Property~(P). 
\end{thm}

\begin{proof}
In each case, Property (P) for $H$ obviously implies Property (P) for $Q$.

Let $X$ be a continuous discrete $H$-set. Corollary \ref{surjh1}, combined with Proposition \ref{linac2}, establishes the surjectivity of the restriction maps $H^1(Q,\ell^2(X^{[Q]})\to H^1(H,\ell^2(X))$ and $\overline{H^1}(Q,\ell^2(X^{[Q]})\to \overline{H^1}(H,\ell^2(X))$. This implies the converse when (P) is either dFH or \odFH.

To deal with Properties FM and FM', let us first check the following claim: if $X$ be a continuous discrete $H$-set, then 
$Z=X^{\rangle Q\langle}$ is not Eymard-amenable. Indeed, otherwise $\ell^2(Z)$ almost has $H$-invariant vectors, since $\pi$ is Kazhdan, this implies that $\ell^2(Z)^Q\neq 0$, but actually $\ell^2(Z)^Q=0$ by Proposition \ref{linac2}, a contradiction.

Now assume that $Q$ has Property FM. Let $X$ be an Eymard-amenable continuous discrete $H$-set. It follows from the above claim that $X^{\langle Q\rangle}$ is Eymard-amenable. It follows that $X^{[Q]}$ is Eymard-amenable, that is, $\ell^2(X^{[Q]})$ almost has invariant vectors as $H$-representation. Noting that this orthogonal $H$-representation factors through $Q$ and using that $\pi$ is Kazhdan, it almost has invariant vectors as $Q$-representation. So $X^{[Q]}$ is an Eymard-amenable $Q$-set. By Property FM for $Q$, it follows that $X^{[Q]}$ has a finite orbit. Since the fibers of $X^{\langle Q\rangle}\to X^{[Q]}$ are finite, it follows that $X^{\langle Q\rangle}$ has a finite orbit. So $H$ has Property FM.

If $Q$ has Property FM' and $X$ is in addition $H$-transitive, the same argument shows that $X^{\rangle Q\langle}=\emptyset$ and $X^{[Q]}$ is finite, thus $X$ is finite and $H$ has Property FM'.
\end{proof}

\begin{prop}\label{ssfm}
If $G$ is semisimple of algebraic type, then it has Property FM; if moreover, no simple factor is non-Archimedean of rank~1, then it has Property dFH.
\end{prop}
\begin{proof}
We start by the second statement.
It is clear that $G$ has Property dFH if and only if $G/(G^\circ\cap\overline{[G,G]})$ has Property dFH. Hence we can suppose $G$ totally disconnected; hence by assumption it has no simple factor of rank~1 and hence has Property~T, which implies Property dFH.

Let us now prove the first assertion.
By Proposition \ref{fmext}, we can deal with the case when $G$ is almost simple (and non-compact). 
Let $X$ be a discrete Eymard-amenable $G$-set.
By the Howe-Moore Property \cite[Theorem 5.1]{HowMoo}, every open subgroup of $G$ is either compact or cocompact, so every stabilizer is either compact or cocompact. If some stabilizer is cocompact, then there is a finite orbit and we are done. Otherwise, all stabilizers are compact and in particular are amenable. By Lemma \ref{multia}, $G$ is amenable, a contradiction.
\end{proof}

Theorem \ref{ipart} follows from the following more general result:

\begin{thm}\label{somtfm}
Let $G=G_1\dots G_n$ be of algebraic type, where $n\ge 2$, and let $\Gamma$ be an irreducible lattice. Suppose that some $G_i$ has Property T. Then
\begin{itemize}
\item $\Gamma$ has Property FM
\item if moreover no $G_j$ is non-Archimedean of rank~1, then $\Gamma$ has Property dFH (and hence FW).
\end{itemize}
\end{thm}
\begin{proof}
We can suppose that the product is direct. If $G_i$ has Property $T$, write $G'=\prod_{j\neq i}G_j$. Then, by Proposition \ref{ssfm}, $G'$ has Property FM, and has Property dFH in case $G$ has no non-Archimedean simple factor of rank 1. The projection $\Gamma\to G'=G/G_i$ is a Kazhdan homomorphism, by Theorem \ref{mbl}. Hence by Theorem \ref{fmres}, $\Gamma$ inherits Property FM, and dFH when applicable, from $G'$.
\end{proof}

\subsection{An infinitely generated countable group with Property FM' and FW'}\label{sl3q}

Let $I$ be a set of primes. Define $\mathbf{A}_I$ as the subring of the product $\prod_{p\in I}\mathbf{Q}_p$ consisting of those $(x_p)_{p\in I}$ such that $x_p\in\mathbf{Z}_p$ for all but finitely many $p$. This is a topological ring, when its subring $\prod_{i\in I}\mathbf{Z}_p$ is prescribed to be endowed with the product topology and to be open in $\mathbf{A}_I$. If $p\in I$, it is convenient to see the factor $\mathbf{Q}_p$ as a (non-unital) subring of $\mathbf{A}_I$. If $I$ is the set of all primes, we just set $\mathbf{A}_I=\mathbf{A}$; this is the ring of {\em adeles}. 

The diagonal ring homomorphism $\mathbf{Z}[I^{-1}]\to \mathbf{R}\times\mathcal{A}_I$ embeds $\mathbf{Z}[I^{-1}]$ as a discrete cocompact subring, whose projection on $\mathcal{A}_I$ has a dense image.

\begin{lem}\label{opena}
Fix $n\ge 2$. Let $H$ be an open subgroup of $\SL_n(\mathbf{A}_I)$. Define $J=\{p\in J:\SL_n(\mathbf{Q}_p)\subset H\}$. Then $H=K\SL_n(\mathbf{A}_J)$ for some compact open subgroup $K$ of $\SL_n(\mathbf{A}_I)$.
\end{lem}
\begin{proof}
First assume that $J=\emptyset$; hence we have to prove that $H$ is compact. Since $H$ is open, there exists a finite subset $F$ of $I$ such that for all $p\notin F$ we have $\SL_n(\mathbf{Z}_p)\subset H$.

\begin{itemize}
\item For $p\notin F$, since $\SL_n(\mathbf{Z}_p)$ is maximal in $\SL_n(\mathbf{Q}_p)$, the projection of $H$ on $\SL_n(\mathbf{Q}_p)$ is either $\SL_n(\mathbf{Z}_p)$ or $\SL_n(\mathbf{Q}_p)$; in the second case, since the intersection is normal in the projection and since $\SL_n(\mathbf{Z}_p)$ is not normal in $\SL_n(\mathbf{Q}_p)$, we deduce that the intersection of $H$ with $\SL_n(\mathbf{Q}_p)$ is all of $\SL_n(\mathbf{Q}_p)$, contradicting that $J=\emptyset$.
\item For $p\in F$, the intersection $H\cap\SL_n(\mathbf{Q}_p)$ has finite index in $\SL_n(\mathbf{Z}_p)$ and hence has finite index in its normalizer $K_p$ (because any open subgroup of $\SL_n(\mathbf{Q}_p)$ distinct from the whole group is compact), which is thus compact. So by the same argument using that the intersection is normal in the projection, we obtain that the projection is contained in $K_p$.
\end{itemize}

Thus $H\subset\prod_{p\in F}K_p\times\prod_{p\notin F}\SL_n(\mathbf{Z}_p)$, which is compact.

Now if $J$ is arbitrary, we have $\SL_n(\mathbf{A}_I)=\SL_n(\mathbf{A}_J)\times \SL_n(\mathbf{A}_{I\smallsetminus J})$. By the previous case, $L=H\cap \SL_n(\mathbf{A}_{I\smallsetminus J})$ is compact and $H=\SL_n(\mathbf{A}_J)\times L$. Defining $K=\left(\prod_{p\in J}\SL_n(\mathbf{Z}_p)\right)\times L$, the subgroup $K$ is open and $H=K\SL_n(\mathbf{A}_J)$.
\end{proof}

\begin{prop}\label{adfm}
For every $n\ge 2$ and every set of primes $I$, the locally compact group $\SL_n(\mathbf{A}_I)$ has Property FM'.
\end{prop}
\begin{proof}
Let $H$ be an open subgroup so that the quotient set $X$ is Eymard-amenable. The group $H$ can be described by Lemma \ref{opena}; let $J$ be given by this lemma. Then modding out by the kernel of the action describes $X$ as the quotient of $\SL_n(\mathbf{A}_{I\smallsetminus J})$ by a compact open subgroup; since it has an invariant mean, it follows that $\SL_n(\mathbf{A}_{I\smallsetminus J})$ is amenable as a topological group, which can occur only if $I=J$, in which case $X$ is reduced to a singleton. This proves Property FM'.
\end{proof}

\begin{cor}\label{slnq}
For every $n\ge 3$ and every set of primes $I$, the group $\SL_n(\mathbf{Z}[I^{-1}])$ has Property FM'. In particular $\SL_n(\mathbf{Q})$ \big(or more generally $\SL_n(\mathbf{Z}[I^{-1}])$ if $I$ is infinite\big) is a countable group with Property FM' but not FM.
\end{cor}
\begin{proof}

The group $\SL_n(\mathbf{Z}[I^{-1}])$ is a lattice in $\SL_n(\mathbf{R})\times\SL_n(\mathbf{A}_I)$, whose projection in the second factor is dense. Moreover, since $n\ge 3$, the group $\SL_n(\mathbf{R})$ has Property T. Thus by Theorem \ref{fmres}, since $\SL_n(\mathbf{A}_I)$ has Property FM' by Proposition \ref{adfm}, it follows that $\SL_n(\mathbf{Z}[I^{-1}])$ has Property FM' as well.

If $I$ is infinite then $\SL_n(\mathbf{Z}[I^{-1}])$ is infinitely generated and hence by Proposition \ref{fmcg} does not have Property FM.
\end{proof}

\begin{rem}
For $n=2$, if $I=\emptyset$, then $\SL_n(\mathbf{Z}[I^{-1}])=\SL_2(\mathbf{Z})$ does not have Property FM'. However, for $I$ nonempty, it sounds plausible that it has Property FM' (this is part of Conjecture \ref{conjm} if $I$ is finite).
\end{rem}

Let us now turn to Property FW'.

\begin{prop}\label{prop620}For every $n\ge 3$ and every set of primes $I$, the locally compact group $\SL_n(\mathbf{A}_I)$ has Property dFH' (and hence Property FW').
\end{prop}
\begin{proof}
Using the description of open subgroups in Lemma \ref{opena}, we are reduced to proving that for every $I$, $n\ge 3$, and compact open subgroup $K$ in $G=\SL_n(\mathbf{A}_I)$, we have $H^1(G,\ell^2(G/K))=0$. If $I=\emptyset$, then $G$ is the trivial group and this is clear; otherwise pick $p\in I$, and define the subgroup $N=\SL_n(\mathbf{Q}_p)$. Since $N$ is a non-compact closed normal subgroup and the representation of $G$ on $\ell^2(G/K)$ is $C^0$, an elementary argument \cite[Lemma 2.9]{CTV} implies that any Hilbertian 1-cocycle that is bounded on $N$ is also bounded on $G$. The boundedness on $N$ is ensured by Property~T of $N$. This shows that $H^1(G,\ell^2(G/K))=0$.
\end{proof}

\begin{cor}
For every $I$ and $n\ge 3$, the countable discrete group $\SL_n(\mathbf{Z}[I^{-1}])$ has Property dFH' and hence FW' (but not Property FW if $I$ is infinite).
\end{cor}
\begin{proof}
Repeat the argument of Corollary \ref{slnq}.

If $I$ is infinite then the group is countable and infinitely generated, hence does not have Property FA and hence does not have Property FW.
\end{proof}

\begin{rem}Unlike the analogue for FM', the condition $n\ge 3$ is here necessary: for every $I$, $\SL_2(\mathbf{Z}[I^{-1}])$ does not have Property FW': if $p\in I$ just use its dense embedding into $\SL_2(\mathbf{Q}_p)$; if $I=\emptyset$ this is clear as well. 
\end{rem}

The following proposition allows to obtain more examples of groups with Property FM'.

\begin{prop}\label{sums}
Let $(S_i)_{i\in I}$ be a family of infinite simple discrete groups with Property FM'. Then the direct sum $S=\bigoplus S_i$ has Property FM'.
\end{prop}
\begin{proof}
Let $H\subset S$ be a subgroup such that $S/H$ is Eymard-amenable and let us show that $H=S$.

If by contradiction the projection $p_i(H)$ of $H$ on some $S_i$ is not surjective, then $S/H$ has an equivariant surjection onto the infinite $S_i$-set $S_i/p_i(H)$, which has no invariant mean by the Property FM' for $S_i$, and hence $S/H$ is not Eymard-amenable. So each projection is surjective, i.e.\ $p_i(H)=S_i$ for all $i$. 

For $i\neq j$, the projection on $S_i\times S_j$ has a surjective projection $P$ on both $S_i$ and $S_j$, so is either the graph of an isomorphism $S_i\to S_j$ or is all of $S_i\times S_j$. In the first case, since $S_i$ is not amenable, there is no $S_i$-invariant mean on $(S_i\times S_j)/P$, and hence $S/H$ is not Eymard-amenable, a contradiction. Thus, for all $i\neq j$ the projection on $S_i\times S_j$ is surjective.

Let $J\subset I$ be maximal such that $\bigoplus_{j\in J}S_j$ is contained in $H$. Let us show that $I=J$, i.e.\ $H=S$. Otherwise, since all projections are surjective, there exists $f\in H$ with support not contained in $J$; we choose $f$ of support of minimal cardinality. The minimality implies that $f$ has support in $I\smallsetminus J$, since otherwise we can modify $f$ on $J$ to reduce its support. Besides, by the definition of $J$, it follows that the support of $f$ has at least two distinct elements $i,j\in I\smallsetminus J$. Since the projection of $H$ on $S_i\times S_j$ is surjective, we can find $g\in H$ such that $g_j=1$ and $g_i$ does not commute with $f_i$. Hence $[g,h]\in H$ and has strictly smaller support, not contained in $J$. This contradicts the definition of $f$.
\end{proof}

\section{Application of Properties FM and FW to groups of bounded displacement permutations}\label{appfm}

Recent attention has been paid to the group $\Wob(\mathbf{Z})$ of bounded displacement permutations of $\mathbf{Z}$ and some more general spaces. A general question, addressed by Juschenko and la Salle \cite{JlS}, is to find general properties of finitely generated subgroups of $\Wob(\mathbf{Z})$, or equivalently to find constraints on homomorphisms from finitely generated groups to $\Wob(\mathbf{Z})$. If more generally $X$ is a discrete metric space with uniformly subexponential growth, in the sense that $\lim_{n\to\infty}\sup_{x\in X}\#B_X(x,n)^{1/n}=1$ where $B_X(x,n)$ is the $n$-ball around $x\in X$, Juschenko and la Salle prove that every homomorphism from a discrete Kazhdan group into $\Wob(X)$ has a finite image. Let us provide the following two extensions of this result.

\begin{thm}\label{fwo}
Let $\Gamma$ be a finitely generated group.
\begin{enumerate}
\item\label{fwwz} if $\Gamma$ has Property FW then every homomorphism $\Gamma\to\Wob(\mathbf{Z})$ has a finite image;
\item\label{fmwx} if $\Gamma$ has Property FM and $X$ is a discrete metric space of uniformly subexponential growth, then every homomorphism $\Gamma\to\Wob(X)$ has a finite image.
\end{enumerate}
\end{thm}
\begin{proof}
Let $u(n)$ be the least upper bound on the size of all closed $n$-balls in $X$.

Note that in both cases, we only have to prove that $\Gamma$ has orbits of bounded cardinality, since it then embeds into an infinite power of some given finite symmetric group, which is locally finite.

Let us first prove (\ref{fmwx}). Let us first check that $\Gamma$ has finite orbits: indeed, the subexponential condition implies that $\Gamma$ preserves a mean on each of its orbits, and Property FM gives the conclusion.

Now let us prove that the orbits has bounded cardinality; otherwise, let $z_n\in X$ belong to an orbit of cardinal $k_n$ with $k_n\to\infty$. Define the radius of $x\in\Gamma z_n$ as the distance $d(x,z_n)$. Let $r_n$ be the largest radius of an element in $\Gamma z_n$. The uniform discreteness implies that $r_n$ tends to infinity, and since generators of $\Gamma$ have bounded displacement, the set of radii of elements of $\Gamma z_n$ is cobounded in $[0,r_n]$, uniformly in $n$. We actually need a more precise statement, namely:

\medskip

\noindent{\em Claim.} There exists $m$ and an sequence of finite subsets $(B_k)$ of $\Gamma$, such that for every $k$ there exists $n_0(k)$ such that for every $n\ge n_0(k)$ there is $\gamma\in B_k$ with $d(z_n,\gamma z_n)\in [k,k+m]$.
 
\medskip 
 
\noindent{\em Proof of the claim.} Let $S$ be a finite symmetric generating subset with identity. Let $m\ge 1$ be an upper bound on the displacement of elements of $S$. So $jm$ is a bound on the displacement of elements of $S^j$. If $S^jz_n=S^{j+1}z_n$, then $S^jz_n=\Gamma z_n$, and therefore $r_n\le jm$. Thus if $j<r_n/m$ then $S^{j+1}z_n$ strictly contains $S^jz_n$. Thus, defining $s_n=\lfloor r_n/m\rfloor$, the cardinal of $S^jz_n$ is $\ge j$ for all $j\le s_n$.

Given $k$, if $n$ is large enough, say $n\ge n_0(k)$, then $s_n\ge u(k)+1$. So  $S^{u(k)+1}z_n$ has cardinal $\ge u(k)+1$. Thus for $n\ge n_0(k)$, $S^{u(k)+1}z_n$ contains an element outside the closed $k$-ball around $z_n$. Pick $\gamma\in S^\ell$ with $\ell$ minimal and $d(\gamma z_n,z_n)> k$; hence $\ell\le u(k)+1$. It follows from minimality that $d(\gamma z_n,z_n)\le k+m$. Thus the claim is proved, with $B_k=S^{u(k)+1}$.

\medskip

Fix a nonprincipal ultrafilter on the positive integers, and let $Y$ be the ultralimit of the sequence of pointed metric spaces $(X,z_n)$. This is the set of sequences $(x_n)$ with $d(x_n,z_n)$ bounded, modulo being at distance $d_\omega$ zero, where $d_\omega((x_n),(x'_n))=\lim_\omega d(x_n,x'_n)$ (the reader can also construct $Y$ as a Gromov-Hausdorff limit). Then $Y$ is a metric space, in which the $n$-balls have cardinal $\le u(n)$, where $u(n)$ is a bound on the cardinals of $n$-balls in $X$, hence has a uniformly subexponential growth. Since elements of $\Gamma$ have a bounded displacement, the mapping $(x_n)\mapsto (\gamma x_n)$, for $\gamma\in\Gamma$, induces an action of $\Gamma$ on $Y$. By the claim, for every $k$ and every $n\ge n_0(k)$, there is $\gamma\in B_k$ such that $d(\gamma z_n,z_n)\in [k,k+m]$. Since $B_k$ is finite, it follows, denoting by $z$ the element $(z_n)$ of $Y$, that for every $k$ there is $\gamma\in B_k$ such that $d(z,\gamma z)\in [k,k+m]$. Thus the $\Gamma$-orbit of $z$ is unbounded. This is a contradiction with the beginning of the proof.

Let us now prove (\ref{fwwz}). Let $\Gamma$ act on $\mathbf{Z}$ by permutations of bounded displacement. If orbits have unbounded (possibly infinite) cardinal, for each $n$ pick an element $z_n$ such that orbit of $\Gamma z_n$ has nonempty intersection with both $\mathopen]-\infty,z_n-n]$ and $[z_n+n,+\infty\mathclose[$. Then consider an ultralimit as previously; we obtain an action on $\mathbf{Z}$ by bounded displacement, so that the orbit of 0 is neither bounded above nor below. On the other hand, $\mathbf{N}$ being commensurated, it is commensurate to an invariant subset, and this is a contradiction with the existence of an orbit accumulating on both $+\infty$ and $-\infty$.
\end{proof}

\begin{rem}
The argument of the proof of Theorem \ref{fwo}(\ref{fwwz}) also shows the following:
\begin{enumerate}
\item
 If a finitely generated group $G$ has a homomorphism into $\Wob(\mathbf{Z})$ with an infinite image, then it also has a homomorphism into $\Wob(\mathbf{Z})$ defining a transitive action on $\mathbf{Z}$.
 \item If $G$ is a finitely generated group with Property FW and $X$ is a Schreier graph of $G$, then for every sequence $(x_n)$ in $X$, any Gromov-Hausdorff limit of the set of pointed graphs $(X,x_n)$ has at most one end; if $X$ is infinite we can thus say that $X$ is ``stably one-ended". An example of a one-ended graph which is not stably one-ended is given by a combinatorial half-line.
\item Let $G$ be a finitely generated group and $H$ a finitely generated subgroup. If $X$ is a connected graph of bounded valency, such that $G$ admits a homomorphism into $\Wob(X)$ such that $H$ has an infinite image, then there exists an ultralimit $X'$ of $X$ and a homomorphism of $G$ into $\Wob(X')$ such that $H$ has an infinite orbit in $X'$ (pick any sequence $(x_n)$ in $X$ such that $\# (Hx_n)\to\infty$ and pick any ultralimit of the family of based graphs $(X,x_n)$). We deduce, for instance, that cyclic subgroups of finitely generated subgroups of $\Wob(\mathbf{Z})$ are undistorted. 
\end{enumerate}
\end{rem}

It would be interesting to have results for locally compact groups as well. But there is a continuity issue when considering the ultralimit. To deal with this continuity issue, it is enough to deal with profinite groups. A simple Baire argument shows that if a compact group acts continuously on a metric space with bounded displacement self-homeomorphisms, then there is an upper bound for the displacement of elements in the whole group.

On the other hand, there exist continuous actions with bounded displacement of topologically finitely generated profinite groups on $\mathbf{Z}$, with orbits of unbounded cardinal. For instance the group $\mathbf{Z}_p$ admits such an action: indeed we can let the generator act using cycles of displacement 2 and length $p^n$ (of the form $(0,2,\dots,2k-2,2k,2k-1,2k-3,\dots,3,1)$ or $(0,2,\dots,2k,2k+1,2k-1,\dots,3,1)$ according to whether $p$ is even or odd). 

The ultralimit construction in the proof of Theorem \ref{fwo} thus provides an action with an infinite orbit, which is necessarily non-continuous.


\begin{thebibliography}{KM98b}


\bibitem[Ber]{Ber} G. Bergman. Generating infinite symmetric groups, Bull. London Math. Soc., 38 (2006) 429--440.

\bibitem[BHV]{BHV} B. Bekka, P. de la Harpe, A. Valette. Kazhdan's Property (T). New Math. Monographs 11, Cambridge Univ. Press, Cambridge 2008.


\bibitem[BL]{BL} B. Bekka, N. Louvet. On a variant of Kazhdan's Property (T) for subgroups of semisimple groups, Ann. Inst. Fourier 47 (4) (1997) 1065--1078.

\bibitem[BO]{BO} B. Bekka, B. Olivier. On groups with Property $(T_{\ell_p})$. J. Funct. Anal. 267 (2014), no. 3, 643--659. 

\bibitem[BPP]{BPP} L. Brailovsky, D. Pasechnik, C. Praeger. Subsets close to invariant subsets for group actions. Proc. Amer. Math. Soc. 123 (1995), no. 8, 2283--2295. 

\bibitem[CCJJV]{CCJJV} P.-A. Cherix, M. Cowling, P. Jolissaint, P. Julg, and A.  Valette. Groups with the {H}aagerup property, Progress in Mathematics, vol. 197. \newblock Birkh\"auser Verlag, Basel, 2001.


\bibitem[CFI]{CFI} I. Chatterji, T. Fernos, A. Iozzi. The median class and superrigidity of actions on CAT(0) cube complexes. 
With an appendix by Pierre-Emmanuel Caprace.
J. Topol. 9 (2016), no. 2, 349--400. 

\bibitem[Cor05]{Cor05} Y. Cornulier. On Haagerup and Kazhdan Properties, Th\`ese en sciences, \'Ecole Polytechnique F\'ed\'erale de Lausanne, No	3438 (2005). doi:10.5075/epfl-thesis-3438

\bibitem[Cor06]{Cor} Y. Cornulier. Relative Kazhdan Property. Annales Sci. \'Ecole Normale Sup. 39 (2006), no. 2, 301-333.

\bibitem[Cor13]{CorFW} Y. Cornulier. Group actions with commensurated subsets, wallings and cubings (2013). ArXiv: math/1302.5982.

\bibitem[CR]{CR} P-E. Caprace and B. R\'emy. Simplicity and superrigidity of twin building lattices.
Invent. Math. 176(1) (2009), pp. 169--221.


\bibitem[CTV]{CTV} Y. Cornulier, R. Tessera, A. Valette. Isometric group actions on Banach spaces and representations vanishing at infinity. Transform. Groups 13 (1) (2008) 125--147.

\bibitem[Eym]{Eym} P. Eymard. Moyennes invariantes et repr\'esentations unitaires. Lecture Notes in Mathematics 300, Springer-Verlag (1972).


\bibitem[GM]{GM} Y. Glasner and N. Monod. Amenable actions, free products and a fixed point property. Bull. London Math. Soc. 39(1), 138--150, 2007.

\bibitem[GN]{GN} R. I. Grigorchuk and V. Yu. Nekrashevych. Amenable actions of non-amenable groups, Zap. Nauchn. Semin. POMI 326 (2005) 85--96.


\bibitem[HM]{HowMoo} R.E. Howe and C.C. Moore.
\newblock Asymptotic properties of unitary representations. J. Funct. Anal. {\bf 32}, 72-96, 1979.


\bibitem[JlS]{JlS} K. Juschenko and M. de la Salle. Invariant means of the wobbling group. Invariant means for the wobbling group. Bull. Belg. Math. Soc. Simon Stevin 22 (2015), no. 2, 281--290. 


\bibitem[LMR]{LMR} A. Lubotzky, Sh. Mozes, M. S. Raghunathan, The word and Riemannian metrics on lattices of semisimple groups, Publ. Math. IHES, No. 91 (2000), 5--53.

\bibitem[Luk]{Luk} J. Lukes. Integral representation theory: applications to convexity, Banach spaces and potential theory. Walter de Gruyter, 2010.

\bibitem[Mar]{Mar} G. Margulis. Discrete subgroups of semisimple Lie groups, Springer Ergebnisse (1991).

\bibitem[Mar81]{Mar81} G. Margulis. On the decomposition of discrete subgroups into amalgams. Selecta Mathematics Sovietica 1 (1981), 197--213.


\bibitem[NRa]{NR} T. Napier, M. Ramachandran. Filtered ends, proper holomorphic mappings of K\"ahler manifolds to Riemann surfaces, and K\"ahler groups.  Geom. Funct. Anal. 17 (2008) 1621--1654.


\bibitem[Sha1]{Sha} Y. Shalom. Rigidity of commensurators and irreducible lattices. Invent. Math. 141, 1--54 (2000).

\bibitem[Sha2]{Shann} Y. Shalom. Rigidity, unitary representations of semisimple groups, and fundamental groups of manifolds with rank one transformation group. Annals of Math., 152(1) (2000) 113--182.


\end{thebibliography}
\end{document}